\documentclass[a4paper,11pt]{amsart} 

\usepackage{amsmath, amsthm, amssymb}
\usepackage{enumerate}
\usepackage{hyperref}

\usepackage{color}

\usepackage[ansinew]{inputenc}
\usepackage{graphicx}

\usepackage[margin=4cm]{geometry}

\newtheorem{theorem}{Theorem}[section]
\newtheorem*{teoA}{Theorem A}
\newtheorem*{teoB}{Theorem B} 
\newtheorem*{teoC}{Theorem C}
\newtheorem*{teoD}{Theorem D}
\newtheorem{corollary}[theorem]{Corollary}
\newtheorem{lemma}[theorem]{Lemma}
\newtheorem{proposition}[theorem]{Proposition}

\theoremstyle{definition}
\newtheorem{definition}[theorem]{Definition}
\newtheorem{remark}[theorem]{Remark}

\newcommand{\eps}{\varepsilon}

\newcommand{\R}{\mathbb R}
\newcommand{\RR}{\R}
\newcommand{\ZZ}{\mathbb Z}
\newcommand{\T}{\mathbb T}

\newcommand{\flow}{\mathcal{R}}
\newcommand{\cW}{\mathcal{W}}
\newcommand{\cF}{\mathcal{F}}
\newcommand{\cH}{\mathcal{H}}
\newcommand{\cK}{\mathcal{K}}
\newcommand{\cG}{\mathcal{G}}
\newcommand{\wcF}{\widetilde \cF}

\newcommand{\fol}{\mathcal{F}}

\newcommand{\hol}{\mathcal{H}}
\newcommand{\oo}{\mathcal{O}}
\newcommand{\cL}{\mathcal{L}}

\newcommand{\cE}{\mathcal{E}}

\newcommand{\mt}{\widetilde M}
\newcommand{\ft}{\widetilde f}

\newcommand{\fbs}{\cW^{cs}_{bran}}
\newcommand{\fbu}{\cW^{cu}_{bran}} 
\newcommand{\wfbs}{\widetilde \cW^{cs}_{bran}}
\newcommand{\wfbu}{\widetilde \cW^{cu}_{bran}}

\newcommand{\TT}{\mathbb{T}}
\newcommand{\HH}{\mathbb{H}}
\newcommand{\wg}{\widetilde \cG}

\title[Ergodicity in hyperbolic 3-manifolds]{Ergodicity of partially hyperbolic diffeomorphisms in hyperbolic 3-manifolds}

\author{Sergio R.\ Fenley} 
\address{Florida State University, Tallahassee, FL 32306}
\email{fenley@math.fsu.edu}

\author{Rafael Potrie} 
\address{Centro de Matem\'atica, Universidad de la Rep\'ublica, Uruguay}
\email{rpotrie@cmat.edu.uy}
\urladdr{http://www.cmat.edu.uy/~rpotrie/}
\curraddr{Institute for Advanced Study, Princeton, NJ 08540.}

\begin{document}

\begin{abstract}
We study conservative partially hyperbolic diffeomorphisms in hyperbolic 3-manifolds. We show that they are always accessible and deduce as a result that every conservative $C^{1+}$ partially hyperbolic in a hyperbolic 3-manifold must be ergodic, giving an affirmative answer to a conjecture of Hertz-Hertz-Ures in the context of hyperbolic 3-manifolds. We also get some results for general partially hyperbolic diffeomorphisms homotopic to the identity and in some isotopy classes on Seifert manifolds.
\bigskip


\medskip

\end{abstract}

\maketitle

\section{Introduction}

The description of the statistical properties of conservative systems is a central problem in dynamics \cite{BDV,KH}.
 Anosov systems are well known to be stably ergodic via the Hopf argument \cite{Anosov}. It took a long time before new examples of stably ergodic diffeomorphisms were shown to exist by Grayson-Pugh-Shub \cite{GPS}: they proved this property for the time one map of the geodesic flow of a surface of constant negative curvature. A key point of the proof was to be able to extend the Hopf argument to foliations whose dimensions do not fill the ambient dimension, using the concept of \emph{accessibility} (see section \ref{s.setting}). This motivated Pugh and Shub to propose their celebrated conjecture about ergodicity and accessibility of partially hyperbolic diffeomorphisms \cite{PS} asserting that accessible ones should be abundant among these systems and that accessibility should imply ergodicity. We refer the reader to \cite{CHHU,Wil} for surveys on these topics. This paper addresses the question of establishing ergodicity of a partially hyperbolic system without need of perturbations by establishing accessibility 
unconditionally among conservative partially hyperbolic diffeomorphisms in certain 3-manifolds or under certain hypothesis.

The main result of this article is the following:

\begin{teoA}
Let $f: M\to M$ be a partially hyperbolic diffeomorphism of class $C^{1+}$ in a closed hyperbolic 3-manifold $M$ which preserves a volume form. Then, $f$ is a $K$-system with respect to volume (and in particular, it is ergodic and mixing). 
\end{teoA}

As shown by \cite{BurnsWilkinson,HHU-Ergo} building on the arguments introduced in \cite{GPS} as well as several new insights, the proof of this result is a consequence
of the following statement; which is then stronger as it requires a weaker assumption than preserving volume and assumes
less regularity. See subsection \ref{ss.ergodicity}. 

\begin{teoB}
Let $f:M \to M$ be a 
$C^1$-partially hyperbolic diffeomorphism of a 
closed hyperbolic 3-manifold, so that the non wandering
set of $f$ is all of $M$.
Then $f$ is accessible. 
\end{teoB}

In \cite{HHU-dim3} (see also \cite{CHHU}) it is shown that the
conclusion of Theorem B follows from
the solution of a problem of geometry of foliations and this is what we solve here. 
In the next section we introduce the setting and explain what are the main technical results we need to prove in order to establish the main theorems. 
Some of the intermediary results to prove Theorem B
hold in the more general setting of partially hyperbolic diffeomorphisms homotopic to the identity in 3-manifolds and this is how
they are proved in this paper. 


In \cite{HHU-Ergo,BHHTU} it is shown that accessible partially hyperbolic diffeomorphisms are open and dense among those with one dimensional center (see \cite{DW} for higher dimensional center, but a weaker topology). Here we treat the problem (started in \cite{HHU-dim3}) of determining which manifolds or homotopy classes of partially hyperbolic diffeomorphisms are \emph{always} accessible or ergodic by topological reasons. Other results in this direction have been obtained, see \cite{HHU-dim3,HUres,GanShi,Hammerlindl}. 
In particular, in \cite[Theorem 1.1]{HRHU} it is proved 
that conservative partially hyperbolic diffeomorphisms homotopic to the
identity of Seifert manifolds are ergodic.
This result follows from our results, but the arguments are quite different and in particular, \cite[Theorem 1.1]{HRHU} does not depend on the classification of partially hyperbolic diffeomorphisms homotopic to the identity in Seifert manifolds that we use.

Notice that it makes sense to prove accessibility without any assumptions on the dynamics of $f$ (such as being non-wandering or conservative) as done in \cite{DW,BHHTU} for manifolds of any dimension.
We provide in section \ref{s.nonconservative} some results in dimension $3$
that allow to treat this case too, namely, when $f$ is leaf conjugate to an Anosov flow.
Previously only in \cite[Section 6.3]{HP-Nil} accessibility in the non-conservative setting was considered in relation with
\cite[Conjecture 2.11]{CHHU}. 
We prove:

\begin{teoC}
Let $f: M \to M$ be a partially hyperbolic diffeomorphism in a 3-manifold $M$
which does not have virtually solvable fundamental group, and suppose that 
$f$ is a discretized Anosov flow.
Then, $f$ is accessible.
\end{teoC}

A discretized Anosov flow is a map $f(x) = \phi_{t(x)}(x)$,
where $\phi_t$ is a topological Anosov flow.
By \cite[Theorem A]{BFFP2} (see also \cite[Theorem A]{BFFP}) the previous theorem gives accessibility for partially hyperbolic diffeomorphisms homotopic to the identity in Seifert manifolds. The ideas behind Theorem B also allow to treat certain isotopy classes in Seifert manifolds which are not dealt with by Theorem C. The following theorem is proved in Section \ref{s.pASeif} (and we refer the reader to that section for some of the definitions in the statement as well as a stronger statement). 

\begin{teoD}
Let $M$ be a Seifert 3-manifold with hyperbolic base and $f: M\to M$ be a partially hyperbolic diffeomorphism, such that its non-wandering set 
is all of $M$, and the induced action of $f$ in the base is pseudo-Anosov. Then, $f$ is accessible. 
\end{teoD}

These results give an affirmative solution to
a conjecture by Hertz-Hertz-Ures \cite[Conjecture 2.11]{CHHU} in hyperbolic 3-manifolds as well as other isotopy classes of diffeomorphisms of some 3-manifolds. See Remark \ref{rem.CDergod} for the ergodic consequences of Theorems C and D.

To perform the proof of Theorem B we need to show a result about quasigeodesic pseudo-Anosov flows in $3$-manifolds 
which may be of independent interest (see Proposition \ref{conjugate}). 

Being partially hyperbolic (via cone-fields) and preserving volume are conditions that are usually given a priori or can be detected by checking only finitely many iterates of $f$, while ergodicity (or mixing) is a chaotic property that involves the asymptotic behaviour of the system, so, this
type of result allows to give precise non-perturbative information of a system by looking at finitely many iterates (c.f. \cite{Potrie-ICM}).

\section{Setting and strategy}\label{s.setting} 

\subsection{Definitions} 

Let $M$ be a closed manifold and $f: M\to M$ a $C^1$-diffeomorphism. One says that $f$ is \emph{partially hyperbolic} if there exists a $Df$-invariant continuous splitting $TM = E^s \oplus E^c \oplus E^u$ into non-trivial bundles such that there is $n>0$ so that for any unit vectors $v^\sigma \in E^{\sigma}(x)$ ($\sigma=s,c,u$) it follows that: 

$$ \|Df^n v^s \| < \min\{1, \|Df^n v^c \|\} \leq \max \{ 1, \|Df^n v^c \|\} < \|Df^n v^u\|. $$

Using an adapted metric, we assume that $n = 1$. We refer the reader to \cite{CP,HP-survey} for basic properties of these diffeomorphisms. 

It is well known \cite{HPS} that the bundles $E^s$ and $E^u$ are uniquely integrable into foliations $\cW^s$, $\cW^u$.

\begin{definition}{(accessibility)}
We say that $f$ is \emph{accessible} if given any two points $x,y\in M$ there exists a piecewise $C^1$ path tangent to $E^s \cup E^u$ from $x$ to $y$. 
More generally two points $z, w$ in $M$ are in the same accessibility class if there is a path as above connecting the points. This is an equivalence 
relation.
\end{definition}

We refer the reader to \cite{CHHU,Wil} for more on the notion
of accessibility.
Notice that the tangent vectors are required to be in $E^s \cup E^u$
and not on $E^s \oplus E^u$.

Recall that a diffeomorphism $f$ is said to be \emph{conservative} if it preserves a volume form. By Poincar\'e recurrence, this implies that such an $f$ has to be \emph{non-wandering} meaning that for every open set $U \subset M$ there exists $n>0$ so that $f^n(U) \cap U$ is not
empty. For all the results of this paper (except to deduce Theorem A from Theorem B) this weaker topological assumption (that $f$ is
non wandering) will be enough. 

We will always work with $M$ being a 3-dimensional manifold. In this setting, a hyperbolic manifold is one which can be obtained as a quotient of $\mathbb{H}^3$ by isometries of the hyperbolic metric.  Contrary to what one might expect, these are well known to be quite abundant thanks to Thurston-Perelman's geometrization theorem.  We emphasize here that many hyperbolic 3-manifolds are known to support partially hyperbolic diffeomorphisms (and even Anosov flows),  see e.g. \cite{Go,FenleyAnosov, FoulonHasselblatt}. 
In fact the Anosov flows analyzed by Foulon and Hasselblatt in 
\cite{FoulonHasselblatt}
are contact, and hence preserve a volume form. 
Notice that Asaoka \cite{As} proved that any transitive Anosov flow
is orbitally equivalent to an Anosov flow preserving a smooth volume
form.
Also any Anosov flow in a hyperbolic $3$-manifold is transitive
\cite{Mosh}.

For a manifold $M$ we will denote by $\mt$ the universal cover, and by $\pi: \mt \to M$ the canonical projection. Whenever an object $X$ is lifted to $\mt$ it will be denoted by $\widetilde X$. 

\subsection{Ergodicity}\label{ss.ergodicity}
A conservative diffeomorphism $f: M\to M$ is \emph{ergodic} if the only $f$-invariant sets in $M$ have zero or total volume. It is \emph{mixing} if it holds that for every measurable subsets $A,B \subset M$ one has that $\lim_n \mathrm{vol}(A \cap f^n(B)) = \mathrm{vol}(A) \mathrm{vol}(B)$ (we are assuming here that we have normalized the volume so that $\mathrm{vol}(M)=1$). 

To prove Theorems A from Theorem B it is enough to apply the following result which is a consequence of the main result of \cite{BurnsWilkinson} (in this setting it also follows from \cite{HHU-Ergo}): 

\begin{theorem}\label{teo.BW}
Let $f: M \to M$ be a conservative partially hyperbolic diffeomorphism of class $C^{1+}$ which is accessible and has center dimension equal to $1$. Then, $f$ is a $K$-system with respect to volume. In particular it is ergodic and mixing. 
\end{theorem} 

Notice that accessibility is stable under taking finite iterates, so for proving Theorem A and B we can always pick a finite iterate of our diffeomorphism. Also, we can take finite lifts if desired as having a finite lift which is accessible implies that the original map is accessible.  

We refer the reader to \cite{Manhe} for definitions and implications of being a $K$-system.
We do not define this notion here, as this paper will be entirely
about geometric notions pertaining to accessibility and related 
objects. We remark that it would be very natural  to ask whether
in this context every conservative partially hyperbolic diffeomorphism is Bernoulli and some indications that this might be the case can be found in \cite{AVW}. This question is beyond the scope of this paper. 

We also remark that in the setting of non-wandering diffeomorphisms, it is a classical result by Brin  \cite{Brin} (see also \cite[Section 5]{CP}) that an accessible partially hyperbolic diffeomorphism which is non-wandering has a dense orbit (i.e. it is transitive). 

\begin{remark}\label{rem.CDergod}
These considerations are also valid in the setting of Theorem C and D, so both those results have their implications on the ergodicity of conservative systems in the setting of Theorem C and D.
\end{remark}

\subsection{Accessibility} 

In \cite{HHU-dim3} (see also \cite{CHHU}) the following result is shown based on the results of \cite{HHU-Ergo} in the specific case of dimension 3 (we remark that the results in \cite{HHU-Ergo} have their non-conservative counterparts in \cite{BHHTU}): 

\begin{theorem}\label{t.HHU}
Let $M$ be a closed 3-manifold and $f: M \to M$ a 
partially hyperbolic diffeomorphism so that the non wandering
set of $f$ is all of $M$.
Suppose that $f$ is not accessible. Suppose that 
there are no tori tangent to $E^s \oplus E^u$. Then 
\begin{itemize}
\item
There exists an $f$-invariant non-empty lamination $\Lambda^{su}$ with $C^1$-leaves tangent to $E^s \oplus E^u$, and so that $\Lambda^{su}$ does
not have any compact leaves.
\item
The completion of each complementary region of $\Lambda^{su}$
is an $I$-bundle, so the lamination $\Lambda^{su}$ 
can be extended to a foliation $\cF$ (not necessarily $f$-invariant nor tangent to $E^s \oplus E^u$) without compact leaves.
\item
Finally the center bundle is uniquely integrable in the completion
$W$ of any complementary region, and in $W$ the center foliation
is made up of compact segments from one boundary component
of $W$ to the other.
\end{itemize}
\end{theorem}

We remark here that the lamination $\Lambda^{su}$ may cover the whole $M$ in which case it would be a foliation and the second and third items become void. 

\begin{remark}\label{r.HHU2}
In fact, under the same assumptions, the conclusions of Theorem \ref{t.HHU} hold for any closed $f$-invariant lamination tangent to $E^s \oplus E^u$ without tori leaves. This follows directly from the proofs that we briefly summarize in the next paragraph.
\end{remark}

The first item of Theorem \ref{t.HHU} 
is proved in \cite{HHU-Ergo} with the hypothesis that
$f$ is conservative or non-wandering. The proof only uses that $f$ is not accessible,
an explicit proof that one only needs that
is done in \cite{CHHU}.
The first statement of the third item is 
proved in Proposition 4.2 of \cite{HHU-dim3}.
It is stated for $f$ conservative, but the proof only needs
that the non wandering set of $f$ is all of $M$.
In fact the specific property that is used is that
recurrent points are dense.
The second statement of the third item is proved during the
proof of Theorem 4.1 of \cite{HHU-dim3}.
Again, all that
is needed is that the non-wandering set of $f$ is all of $M$.
Finally, 
the second statement in the third item
then implies the second item.


\begin{remark}\label{rem.toriimplysolv} It follows from \cite{HHU-tori} that if there is a torus tangent to $E^s \oplus E^u$ then it must be incompressible and the manifold $M$ has virtually solvable fundamental group. In particular when $M$ is a hyperbolic 3-manifold there are not tori tangent
to $E^s \oplus E^u$. 
\end{remark}

\begin{remark}
For the remainder of the article, unless otherwise stated,
$M$ will have dimension $3$. 
\end{remark}

\subsection{Strategy of the proof} 
The core of the proof of Theorems B, C and D is already present in Theorem B which is also the hardest. This is why we will concentrate mainly in Theorem B and explain at the end of the paper how the proofs adapt to give Theorems C and D.

The proof of Theorem B (which implies Theorem A) is based on Theorem \ref{t.HHU} which reduces the study to the case where there is an $f$-invariant lamination (which can be completed into a foliation without compact leaves) whose leaves are saturated by stable and unstable manifolds. 
In other words we want to show that this is not possible 
for $M$ hyperbolic.

When the manifold is hyperbolic, up to taking an iterate one can assume that $f$ is homotopic to the identity (c.f. subsection \ref{ss.Mostow}).
 Therefore it admits what is called
a \emph{good lift} $\ft$ at a bounded distance from the identity and commuting with deck transformations (again see subsection \ref{ss.Mostow}). 
We stress that 
a lot of the analysis will be done in the more general 
setting of homeomorphisms homotopic to
the identity.

The study of  \emph{taut} foliations invariant under a diffeomorphism $f$
homotopic to the identity
in general 3-manifolds was started in \cite{BFFP,BFFP2,BFFPann}; 
 where a general 
dichotomy was obtained (see subsection \ref{ss.clasifres}): 
Let $\ft$ be a good lift. Then either 

\begin{itemize}
\item there is a non empty, leaf saturated,
closed set in $M$, \ and
 whose lift to $\mt$ is leafwise fixed by $\ft$; 

\item
or, the foliation is $\mathbb{R}$-covered and uniform, and $\ft$ acts as a translation on the leaf space of the foliation lifted to $\mt$.
\end{itemize}

\noindent
This works for any taut foliation and not just those
associated with a partially hyperbolic diffeomorphism.

Then one 
applies the dichotomy above to the lamination $\Lambda^{su}$
obtained from Theorem \ref{t.HHU}.
One analyzes each case of the dichotomy separately to get a contradiction
to non accessibility of $f$ partially hyperbolic in
a hyperbolic $3$-manifold. 
A contractible fixed point is a fixed point of $\ft^k$ with $k$ not 
$0$ and $\ft$ a good lift.
The case 
where some of the leaves are fixed by $\ft$
requires a general result of \cite{BFFP2} stating that $f$ cannot have contractible periodic points, and this 
then allows to perform arguments similar to those dealing with the \emph{doubly invariant case} of \cite{BFFP,BFFP2} which we do in section \ref{s.FL}. 
This case can be dealt with for general partially hyperbolic
diffeomorphisms of $3$-manifolds that are homotopic to the
identity. 

The full translation case (dealt with in section \ref{s.TR}) is specific to hyperbolic manifolds as it uses the existence of a regulating pseudo-Anosov flow (see subsection \ref{ss.reg}) to show that the partially hyperbolic diffeomorphism needs to be leaf conjugate to a topological Anosov flow. The ideas to show this are similar to the ones appearing in the classification of dynamically coherent partially hyperbolic diffeomorphisms in hyperbolic 3-manifolds (this is done in subsection \ref{ss.Triple}). Then one uses some well known properties of topological Anosov flows (see subsection \ref{ss.disAnosov}) and some properties of regulating pseudo-Anosov flows (in particular, we mention Proposition \ref{conjugate} which may be of independent interest) to get a contradiction (cf. subsection \ref{ss.DAF}).
One much easier subcase of the translation case also works for
general partially hyperbolic diffeomorphisms homotopic to the identity when $\pi_1(M)$ is not virtually solvable (cf. Remark \ref{rem.toriimplysolv}).

\section{Preliminaries and reductions} 
\subsection{Mostow rigidity and good lifts}\label{ss.Mostow}
Let $f: M \to M$ be a homeomorphism of a closed hyperbolic 3-manifold. 
Mostow rigidity \cite{Mostow}, implies that any homeomorphism of a closed hyperbolic manifold (in fact in any dimension $\geq 3$)
is homotopic to an isometry. In any closed manifold an isometry has an iterate close to the identity. This implies that in a hyperbolic manifold, every homeomorphism has an iterate homotopic to the identity (see \cite[Appendix A]{BFFP}).

\begin{definition}{(good lift)}
Let $f$ be a homeomorphism of a closed manifold.
We say that a lift $\ft: \mt \to \mt$ is a \emph{good lift} of $f$,
 if $\ft$ is at bounded distance from the identity (i.e. there exists $K>0$ so that $d(x,\ft(x)) < K$ for all $x \in \mt$), and $\ft$ 
commutes with all deck transformations. 
\end{definition}

\begin{remark}
In fact it is easy to prove that the first condition is
implied by the second one.
We state the definition of good lifts this way, to emphasize
that both properties are used throughout the analysis.
\end{remark}

\begin{lemma}
If $f: M \to M$ is homotopic to identity, then
a lift to $\mt$ of a homotopy from the identity to
$f$ in $M$ provides a good lift of $f$.
\end{lemma}

In conclusion: 

\begin{proposition}\label{p.mostow}
Let $f$ be a diffeomorphism of a closed hyperbolic 3-manifold $M$. Then, there is an iterate of $f$ which admits a good lift to $\mt$. 
\end{proposition}

Notice that in hyperbolic $M$ the good lift is unique (but we will not use this fact). 

\subsection{Foliations without compact leaves}\label{ss.reg}  

A foliation $\cF$ on a closed 3-manifold $M$ will mean a continuous 
two dimensional foliation with leaves of class $C^1$ and tangent to a continuous distribution (foliations of class $C^{0,1+}$ according to \cite{CandelConlon}). We will work with foliations without compact leaves. In particular thanks to results of Novikov, Palmeira and others (see \cite{Calegari,CandelConlon}) it follows that the fundamental group of each leaf injects in $\pi_1(M)$ and therefore every leaf lifts to a plane in the universal cover $\mt$ which necessarily is diffeomorphic to $\RR^3$. We denote by $\widetilde \cF$ the foliation lifted to $\mt$.

For such a foliation, there cannot be a nullhomotopic closed
curve transverse to the foliation, and this implies that the leaf space 

$$\cL_{\widetilde \cF} \ \ := \ \ \mt /_{\widetilde \cF}$$ 

\noindent
is a one-dimensional simply connected manifold (possibly non-Hausdorff). If the leaf space is Hausdorff, then it is homeomorphic to $\RR$ and in this case we say the foliation $\cF$ is $\RR$-\emph{covered}. 

We also consider the following geometric condition: suppose that given any two leaves $L, L' \in \widetilde \cF$ the Hausdorff distance between the leaves is bounded (by a bound that obviously
depends on $L$ and $L'$). In this case we say that the foliation is \emph{uniform} \cite{Th2}. We notice that the distance one needs to consider in $\mt$ is relevant as this is a geometric condition, but since $M$ is compact, any distance which is $\pi_1(M)$-equivariant will work (in particular, when one works with $M$ hyperbolic, one can consider the standard hyperbolic metric on $\mt = \HH^3$). 

If $\cF$ is a (transversally orientable) foliation on a closed 3-manifold $M$ and $\Phi_t : M \to M$ is a flow on $M$, we say that $\Phi$ is \emph{regulating} for $\cF$ if the orbits of $\Phi$ 
are transverse to $\cF$ and when lifted to the universal cover, it holds that for every $x \in \mt$ and 
$L \in \widetilde \cF$ it follows that there exists $t \in \RR$ such that 
$\widetilde \Phi_t(x) \in L$. (Notice that under the assumption that $\cF$ has no closed 
leaves, this implies that the orbit of $x$ cannot intersect $L$ more than once
by Novikov's theorem \cite{CandelConlon}. Here $\widetilde \Phi$ is
the flow $\Phi$ lifted to $\mt$.)

If $\cF$ admits a regulating flow, it follows that $\cF$ is $\RR$-covered. In hyperbolic manifolds, one has the following very interesting strong converse to this proved  by Thurston \cite{Th2} (see also \cite{CalegariPA, Fen2002}): 

\begin{theorem}[\cite{Th2,CalegariPA,Fen2002}] \label{teo.regul}
Given an $\RR$-covered uniform (transversally orientable) foliation 
$\cF$ on a closed hyperbolic manifold $M$, there exists a regulating pseudo-Anosov flow $\Phi_\cF$ to $\cF$. 
\end{theorem}

\begin{figure}[ht]
\begin{center}
\includegraphics[scale=0.70]{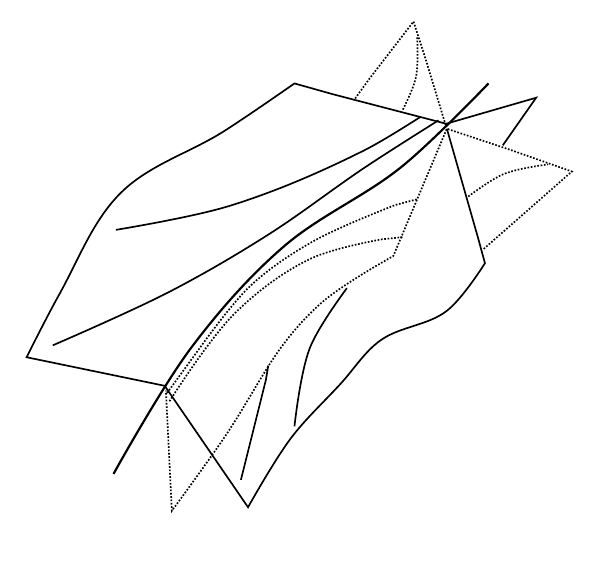}
\begin{picture}(0,0)
\end{picture}
\end{center}
\vspace{-0.5cm}
\caption{{\small The local figure close to a $p$-prong with $p=3$.}\label{f.pAnosov}}
\end{figure}

We recall that a flow $\Phi: M \to M$ with $C^1$ flow lines is said to be \emph{pseudo-Anosov} 
if  there are 2-dimensional, possibly singular, foliations,
$\cG^s$ (stable), $\cG^u$ (unstable) which are flow invariant which verify that:
\begin{itemize}
\item
In a leaf of $\cG^s$ all orbits are forward asymptotic,
and backwards orbits diverge from each other; the analogous
opposite statement for leaves of $\cG^u$;
\item
Singularities,  if any, are of $p$-prong type, with $p \geq 3$.
This means that they are periodic orbits, finitely many, and
locally the stable leaf is a $p$-prong in the plane times
an interval (see figure \ref{f.pAnosov}).
\end{itemize} 

\subsection{Discretized Anosov flows}\label{ss.disAnosov}

A diffeomorphism $g:M \to M$ is said to be a \emph{discretized Anosov flow} if there exists a $g$-invariant $1$-dimensional continuous foliation $\cF^c$ which supports a topological Anosov flow $\{\Phi_t : M\to M\}$
(that is, $\Phi_t(x)$ is always in the same $\cF^c$ leaf
as $x$ for any $x$ in $M$ and any real $t$);
 and such that there exists a continuous function $\tau: M \to \RR_{>0}$ so that $g(x) = \Phi_{\tau(x)}(x)$.

Recall that a \emph{topological Anosov flow} is a pseudo-Anosov flow which has no singular periodic orbits. 

 \begin{remark} A partially hyperbolic diffeomorphism which is a discretized Anosov flow will be necessarily dynamically coherent and the center leaves will correspond to the orbits of the flow (see \cite[Appendix G]{BFFP}).
 \end{remark}

We will use the following result from topological Anosov flows (cf. \cite[Proposition D.4]{BFFP}): 

\begin{theorem}[\cite{FenleyTAMS, Fenley-CMH}]\label{teo.AnosovFlows}
Let $\Phi$ be a topological Anosov flow on a 3-manifold $M$ which is regulating for a uniform $\RR$-covered foliation $\cF$ in $M$. Then $M$ has 
virtually solvable fundamental group, and $\Phi_t$ is orbit equivalent to a suspension flow of a linear automorphism of $\TT^2$. 
\end{theorem}

It is interesting to compare the case of suspensions (which yield solvable fundamental group of the 3-manifold) with the thorough study of accessibility classes in such manifolds performed in \cite{Hammerlindl}.

\subsection{Dichotomies for foliations and laminations}\label{ss.clasifres}
Here we collect some of the results that are needed from \cite[Section 3]{BFFP} as well as some properties that we will use. 

We will use the following result: 

\begin{proposition}[Dichotomy for lifted foliations \cite{BFFP}]
\label{prop.dich}
Let $h: M \to M$ be a homeomorphism of a $3$-manifold  which is
homotopic to the identity.
Let $\cF$ be a foliation without compact leaves in $M$ preserved by $h$. 
Let $\widetilde h$ be a good lift of $h$.
Suppose that there is a leaf $L$ of $\widetilde \cF$ such that $\widetilde h(L) \neq L$.
Then there are two possibilities:
\begin{enumerate}
\item The foliation $\cF$ is $\RR$-covered and uniform and $\widetilde h$ acts as a translation on the leaf space of $\widetilde \cF$. 
\item The leaves $L$ and $\widetilde h (L)$ bound a region
in $\mt$ with closure $U$, and the foliation
$\widetilde \cF$ has leaf space homeomorphic to a closed
interval in $U$. In addition if
$V= \bigcup_{n\in \ZZ} \widetilde h^n(U)$ then each leaf of $\partial V$ is fixed by $\widetilde h$ and the set $V$ is precisely invariant (meaning that if $\gamma \in \pi_1(M)$ verifies that $\gamma V \cap V \neq \emptyset$ then $\gamma V = V$).  
\end{enumerate}
\end{proposition}

For an account on complementary regions to laminations, we refer the reader to \cite[Section 5.2]{CandelConlon}. 

\begin{lemma}\label{l.homot}
Let $f:M\to M$ be a diffeomorphism which
 preserves a lamination $\Lambda$ with $C^1$ leaves,
and such that each completion of a complementary region of $\Lambda$
is an $I$-bundle.
Then, there exists a homeomorphism $h: M \to M$ homotopic to $f$ which coincides with $f$ in $\Lambda$,  and there exists
a $h$-invariant foliation $\fol$ which extends $\Lambda$. 
\end{lemma}

\begin{proof}{}
Since the completion of complementary regions are $I$-bundles, one can extend $\Lambda$ to a foliation $\fol_0$
so that in the closure of each complementary region of $\Lambda$
the foliation $\fol_0$ is a product foliation. 
The completion being an $I$-bundle means that the completion
is diffeomorphic to $F \times [0,1]$ where $F$ is 
surface. Extend the foliation so that the leaves in
$F \times [0,1]$ are $F \times \{ t \}$.
In fact this can be done so that $\fol$ is a $C^{0,1+}$ 
foliation \cite{CandelConlon}.

Choose a one dimensional foliation $\eta$ with $C^1$ leaves,
transverse to $\fol_0$, and 
 so that in each closure of a complementary component
of $\Lambda$ the foliation $\eta$ provides an $I$-bundle structure. 
This is standard, see \cite{CandelConlon}, but we provide the
main ideas: choose a one dimensional foliation $\cG$ with
$C^1$ leaves transverse to
$\fol_0$. Let $U$ be a complementary region and $V$ its completion.
If $V$ is compact, it is obvious how to do this.
Otherwise
$V = K \cup B$ where $K$ is compact and in $B$ the
distance between the upper and lower boundary is very
small. In particular if this distance is very small, then
every leaf of $\cG$ in $B$ goes from the lower boundary
of $B$ to the upper boundary of $B$. Then $K = C \times [0,1]$
where $C$ is a compact surface. In addition in $\partial C \times
[0,1]$ one can assume the one dimensional foliation $\cG$ is
vertical. Since $\cF$ is a product two dimensional
foliation in $C \times [0,1]$ it is
easy to prove that every leaf of $\cG$ in $C \times [0,1]$ 
is a compact segment from $C \times \{ 0 \}$ to 
$C \times \{ 1 \}$.

\vskip .1in
We will now define $\fol$ and the map $h$. Let $A$ be a complementary region. 

First consider the case that the orbit of $A$ is infinite, in
other words $A$ is not periodic. Then for each $n$ define
a foliation $\fol$ in $f^n(A)$ as $f^n(\fol_0 | A)$.
Let $h = f$ in the $f$ orbit of $A$. Clearly this satisfies
the conditions of the Lemma in these regions.

The other case is that $A$ is periodic, let $n$ be the
smallest positive integer so that $f^{n}(A) = A$.
A priori there could be infinitely many such regions.
For each complemetary region $B$, then $B$ union its
two boundary leaves $F_0, F_1$ is an $I$-bundle with
a product foliation.
We consider the case that $F_0$ is non compact, the
other case being simpler.
Then any such $C = B \cup F_0 \cup F_1$ is the union
of a compact core and non compact parts \cite{CandelConlon}.
By choosing the non compact parts very thin, we can ensure
that in such a non compact part of $C$, then
$f(\fol_0)$ is transverse
the $I$-fibers in the particular $I$-bundle.
Finally only finitely many complementary regions have
a compact part with thickness bigger than a given $\delta > 0$.
Hence except for these finitely many compact regions in these
finitely many regions then $f(\fol_0)$ is transverse to the
$I$-fibers.

Now consider a specific periodic $A$ with period $n > 0$.
Let the boundary leaves be $F_0, F_1$.
In each $0 \leq i < n$ define $\fol$ in $f^i(A)$ to be 
the image of $\fol_0 | A$ by $f^i$. 
Now we define $h$. For each $0 \leq i < n-1$ (here
$i \leq n-1$ and not $i \leq n$), let $h$ to be equal
to $f$ in these complementary regions.
Now consider $f^n(A) = A$. The foliation $\fol$ in $A$ is
already defined to be equal to $\fol_0$ here.
Now we will define $h$ in $f^{n-1}(A)$. By hypothesis
$f(\fol)$ is transverse to the $I$-fibration in the 
non compact part of $f^{n-1}(A \cup F_0 \cup F_1)$.
So we can just homotope $f$ to a map $h$ along the $I$-fibers
in the non compact part of $A \cup F_0 \cup F_1$,
so that the new map $h$ sends $\fol | f^{n-1}(A)$
to $\fol | A$, restricted to the non compact part.
In fact we can do this so that $f^n$ fixes
every leaf in the non compact part.
The remaining set in $A$ is compact, where the foliations
are products. Here again $f$ can be homotoped
to $h$ satisfying that $h$ sends $\fol | f^{n-1}(A)$ to
$\fol | A$. 

This finishes the proof of the Lemma.
\end{proof}

As a consequence, we obtain: 

\begin{corollary}{}\label{lami} Suppose that $f$ is a diffeomorphism homotopic to the identity.
Suppose that $f$ preserves a lamination $\Lambda$ with $C^1$ leaves,
and that each completion of a complementary region of $\Lambda$
is an $I$-bundle.
Then, any good lift $\ft$ of $f$ verifies that
\begin{itemize}
\item either there is a compact sublamination $\Lambda'$ 
of $\Lambda$, whose leaves are fixed by $\ft$ when lifted to $\mt$, or
\item let $\fol$ be a foliation which extends
$\Lambda$ as in Lemma \ref{l.homot}. Then $\fol$ is $\RR$-covered and uniform, and $\ft$ acts as a translation on the leaf space
 of $\widetilde \Lambda$ as a subset of the leaf space of $\widetilde \cF$. We sometimes refer to this as
$\ft$ acts as a translation on $\widetilde \Lambda$.
\end{itemize}
\end{corollary}

\begin{proof}{}
First we apply Lemma \ref{l.homot} to obtain $\fol$ and $h$.
At this point, the only property that needs to be proved
is that if $\widetilde h$ is a good lift of $h$,
and $\widetilde h$ fixes a leaf of $\widetilde \fol$ then $\ft$ fixes a leaf of $\widetilde \Lambda$. 
This is because we also proved in \cite[Proposition 3.7]{BFFP} that the
set of leaves of $\widetilde \Lambda$ fixed by $\widetilde h$
is a closed subset of leaves, and hence projects to a sublamination
of $\Lambda$ in $M$.

Suppose that there is a leaf $L$ of $\widetilde \fol$
that is fixed by $\widetilde h$. If $L$ is in $\widetilde \Lambda$,
then $\ft$ also fixes $L$ and we are done.

Otherwise $\pi(L)$ is in a complementary region $W$ of $\Lambda$.
The completion of $W$ is a product foliated $I$-bundle.
Since $\ft$ preserves $\widetilde W$ and $\ft$ preserves
$\widetilde \Lambda$, it now follows that
$\ft$ fixes each of the two boundary leaves of $\widetilde W$.
This finishes the proof of the corollary.
\end{proof}

\begin{remark}\label{r.boundedmovement} 
In the first case, where there is a compact lamination $\Lambda'$ of leaves whose lifts are fixed by $\ft$, it follows that there is a uniform bound on the displacement of points inside the leaves by $\ft$. This is because $\ft$ is bounded distance from the identity and the lamination $\Lambda'$ is compact. See \cite[Section 3]{BFFP}.
\end{remark}


\subsection{Dynamical coherence and incoherence}

A partially hyperbolic diffeomorphism is {\em dynamically coherent}
if there are foliations $\fol^{cs}$ and $\fol^{cu}$ that are
$f$-invariant, and everywhere tangent to $E^{cs} = E^s \oplus E^c$
and $E^{cu} = E^c \oplus E^u$.
These are called the center stable and center unstable foliations.
In that case the intersection of these two foliations
is a one dimensional foliation tangent to $E^c$, called the center
foliation.
But this condition is not always satisfied, there are several 
recent examples which are not dynamically coherent
\cite{HHU-noncoherent,BGHP}. See also \cite[Section 4]{HP-survey} for more information and context. 

However, under very general orientability conditions
there are ``generalized"
foliations tangent to $E^{cs}$ and $E^{cu}$:

\begin{theorem}[Burago-Ivanov \cite{BI}]\label{teo-BI}  
Let $f$ be a partially hyperbolic diffeomorphism of a $3$-manifold $M$ so that the bundles  $E^s, E^c, E^u$
are orientable and $Df$ preserves these orientations.
Then, there are collections $\cW_{bran}^{cs}$ and $\cW_{bran}^{cu}$ of complete immersed surfaces tangent respectively to $E^{cs}$ and $E^{cu}$ satisfying the following properties: 
\begin{itemize}
\item every point $x \in M$ belongs to at least one surface (called \emph{leaf}) of $\cW^{cs}_{bran}$ (resp. $\cW^{cu}_{bran}$),
\item  the collection is $f$-invariant: if $L \in \cW^{cs}_{bran}$ (resp. $L \in \cW^{cu}_{bran}$) then $f(L) \in \cW^{cs}_{bran}$ (resp. $\cW^{cu}_{bran}$), 
\item different leaves of $\cW^{cs}_{bran}$  (resp. $\cW^{cu}_{bran}$) do not topologically cross,
\item if $x_n \to x$ and $L_n$ are leaves of $\cW^{cs}_{bran}$ (resp. $\cW^{cu}_{bran}$) containing $x_n$ then up to 
subsequence  $L_n \to L$ for a leaf $L$
of $\cW^{cs}_{bran}$ (resp. $L \in \cW^{cu}_{bran}$). 
\end{itemize}
Moreover, for every $\eps>0$ there are approximating foliations $\cW^{cs}_\eps$ and $\cW^{cu}_\eps$ tangent to subspaces making angle less than $\eps$ with $E^{cs}$ (resp. $E^{cu}$). In addition
there are continuous maps $h^{cs}_\eps: M\to M$ and $h^{cu}_\eps : M \to M$ at $C^0$-distance less than $\eps$ from identity so that $h^{cs}_\eps$ (resp. $h^{cu}_\eps$) maps an arbitrary leaf of 
of $\cW^{cs}_\eps$ (resp. $\cW^{cu}_\eps$) into a leaf
of $\cW^{cs}_{bran}$ (resp. $\cW^{cu}_{bran}$), by a local diffeomorphism. 
\end{theorem}

The collections of surfaces $\fbs, \fbu$ are called the
center stable and center unstable branching foliations.
The individual surfaces are called the leaves.
If say $\fbs$ does not form a foliation, then $h^{cs}_{\eps}$ is 
not a local diffeomorphims for any $\eps>0$: there are open sets
where it collapses leaves transversely.
When $h^{cs}_{\eps}$ is restricted to an arbitrary leaf 
of $\fbs$, then $h^{cs}$ is 
a local diffeomorphism. Even then it may not be a global diffeomorphism
because leaves of $\fbs$ may self intersect, forming branching locus.

\vskip .1in
When lifted to the universal cover the foliations $\wfbs,\wfbu$ have
``leaf spaces" which are one dimensional, simply connected,
but possibly non Hausdorff, just as in the case of foliations.
Since a point does not determine the leaf it is on, this is
not immediate as in the case of actual foliations. But
one can approximate $\fbs$ by actual foliations and the
above follows, using such an approximation. We refer the
reader to \cite[Section 3]{BFFP2} for a detailed treatment.
A good lift acts on $\mt$ and hence acts as homeomorphisms on the
leaf spaces of $\wfbs$ and $\wfbu$.

\begin{remark}
One can always take a finite lift $g$ of an iterate of $f$
so that $g$ satisfies the condition of Theorem \ref{teo-BI},
and hence $g$ admits branching foliations as in 
Theorem \ref{teo-BI}.
\end{remark}

\subsection{Classification results}

We need the following classification result from \cite[Theorem B]{BFFP2} for hyperbolic manifolds\footnote{The results hold for general partially hyperbolic diffeomorphisms homotopic to the identity, we refer the reader to \cite[Section 2]{BFFP2} for such statements including some stronger results even in the hyperbolic manifold case.}. 

\begin{theorem}[\cite{BFFP2}]\label{teo.clasif}
Let $f: M \to M$ be a partially hyperbolic diffeomorphism of a hyperbolic 3-manifold $M$ with $f$ homotopic to the 
identity, and let $\ft$ be a good lift of $f$.  Suppose that $f$ preserves branching foliations 
$\fbs, \fbu$. Then either
\begin{enumerate}
\item $f$ is a discretized Anosov flow, and in particular $f$ is
dynamically coherent; 
\item\label{it.2} or $f$ is not dynamically coherent
and $\ft$ is a \emph{double translation}. That is, $\ft$ is a
translation on each of the leaf spaces of $\wfbs$ and $\wfbu$.
\end{enumerate}
\end{theorem}

By a translation on the leaf space of $\wfbs$ we mean the following:
the foliation $\fbs$ is approximated arbitrarily well by
a Reebless foliation. This implies that one can also define
the leaf space of the branched foliation $\wfbs$ and that it
is a simply connected one manifold, possibly non Hausdorff.
If it is Hausdorff, then it is homeomorphic to $\R$.
So $\ft$ acting as a translation on the leaf space of
$\wfbs$, means that $\ft$ does not fix any leaf of $\wfbs$,
this leaf space is homeomorphic to $\R$, and $\ft$ acts
as a translation on $\R$.  

We remark that there are no known examples of partially hyperbolic diffeomorphisms of the form (\ref{it.2}) in $M$ hyperbolic. On
the other hand it is shown in \cite{BFFP2} that
 such hypothetical examples should share several features with the actual recent examples discovered in \cite{BGHP}. 
We stress that this refers to a hypothetical possibility
in hyperbolic $3$-manifolds in \cite{BFFP,BFFP2},
and to actual examples in some Seifert manifolds
$M = T^1 S$ in \cite{BGHP}. In fact, in a paper in preparation \cite{FP} we show the following:
that every partially hyperbolic diffeomorphism in
a Seifert manifold, in the isotopy class of those described in \cite{BGHP} (when the action in the base is pseudo-Anosov), as well as every 
hypothetical double translation in a hyperbolic 3-manifold should look very similar to each other
from the point of view of their topological structure. 

We also mention the following dynamical consequence that will be useful for this paper which is \cite[Theorem 1.3]{BFFP2}.
Recall that a fixed point of $\ft^k$ where $\ft$ is a good lift of $f$ is called a \emph{contractible periodic point.}
The foliation
$\fbs$ (or $\fbu$) is $f$-{\emph {minimal}} if $M$ is the only
non empty, closed, $\fbs$ (resp $\fbu$) saturated set that is $f$-invariant.


\begin{theorem}[\cite{BFFP2}]\label{teo.nocontractible}
Let $f: M\to M$ be a partially hyperbolic diffeomorphism homotopic to identity in $M$ 
so that $\pi_1(M)$ is not virtually solvable.
Suppose that there are $f$-invariant branching foliations
$\fbs, \fbu$. Suppose that either $\fbs$ or $\fbu$ is
$f$-minimal. Then $f$ has no contractible periodic points.
\end{theorem}

This result does not need $M$ hyperbolic. On the other
hand, when $M$ is hyperbolic one does not need to assume that the branching foliations exist, or are $f$-minimal as this can be achieved by a finite cover.

\section{Proof of Theorem B}
\label{section4}
Consider a partially hyperbolic diffeomorphism $f: M\to M$ where $M$ is a closed hyperbolic 3-manifold. The goal is to prove that $f$ is
accessible.
Since the strong stable and unstable foliations are the same for iterates, it is no loss of generality to assume that $f$ is homotopic to identity and admits a good lift $\ft$ to $\mt$  (c.f. Proposition \ref{p.mostow}). We will assume that $f$ is non-wandering (recall that if $f$ is non-wandering, so are its iterates\footnote{This is because for a homeomorphism of a compact metric space whose non-wandering set is all the space, the set of recurrent points is dense.}). 

For most of the analysis we will not need to assume that $M$ is hyperbolic, just that $f$ admits a good lift $\ft$ to $\mt$ and that $\pi_1(M)$ is not virtually solvable. We will state explicitly the place where we use that $M$ is hyperbolic. 

We will proceed by contradiction assuming that $f$ is not accessible and appeal to Theorem \ref{t.HHU} that provides an $f$-invariant lamination $\Lambda^{su}$ whose leaves are tangent to $E^s \oplus E^u$. Moreover, this
lamination has no closed leaves (because $\pi_1(M)$ is not virtually solvable, c.f. Remark \ref{rem.toriimplysolv}) and can 
be completed to a foliation $\cF$ so that the completion $U$ of any
complementary region of $\Lambda^{su}$ is a product 
foliated $I$-bundle. In fact, from Theorem \ref{t.HHU} we know that the open set $U \cong F \times (0,1)$ (the completion of $U$ being the addition of $F\times \{0\}$ and $F \times \{1\}$) and: 
\begin{itemize}
\item[(1)] the foliation $\fol$ in $U$ is a product
foliation, and, 
\item[(2)] the center bundle in $U$ is uniquely
integrable and the center one dimensional foliation is also a
product foliation in $U$. 
\end{itemize}

This is the only place where the fact that $f$ is non-wandering will be used. (We do not really need in (2) the center to be uniquely integrable in those regions, just that center curves join both sides of the complementary region, but it is helpful for the presentation of the arguments.)

We denote by $\widetilde \Lambda^{su}$ and 
$\widetilde \cF$ the lifts of these objects to $\mt$. The lifted lamination 
$\widetilde \Lambda^{su}$ is invariant under $\ft$. 
In addition there is $h$ homotopic to $f$ so that $h = f$
in $\Lambda$ and $h$ preserves $\fol$. Let $\widetilde h$ be a corresponding
good lift of $h$ to $\mt$.

We can apply Corollary \ref{lami} to the lamination $\widetilde \Lambda^{su}$.
This separates the study into two cases, which we will deal with separately. 

\vskip .15in
\noindent
{\bf {Case I}} $-$ 
There  is a leaf of $\widetilde \Lambda^{su}$ which is fixed by the good lift $\ft$. 

By Corollary \ref{lami} this is equivalent to $\widetilde h$ 
fixing a leaf of $\widetilde \fol$.
Here we will prove the following:

\begin{proposition}\label{p.fixedleaf}
Consider a good lift $\ft$ of a partially hyperbolic diffeomorphism $f: M \to M$, 
homotopic to the identity, so that $\pi_1(M)$ is not
virtually solvable.
Then $\ft$ cannot leave invariant a leaf of $\widetilde \Lambda^{su}$. 
\end{proposition} 

The proof of this proposition is deferred to section \ref{s.FL}. 
This will prove that Case I cannot happen. We mention here that there is a simpler proof of Proposition \ref{p.fixedleaf} communicated to us by 
Andy Hammerlindl \cite{HRHU} using an old result from Mendes
\cite{Mendes}. Both proofs work without any assumption on the topology of $M$ other than having that $\pi_1(M)$ is not solvable.
We present our proof here in order to have a complete proof of Theorems 
B and C. In addition our proof exemplifies, in a simplified context, some of the technical results of \cite[Section 4]{BFFP} (which are simplified by the fact that 
leaves are subfoliated by strong stables and unstables while in \cite{BFFP} they are foliated by strong stables and have some center foliation whose dynamics is harder to deal with). 

\vskip .15in
\noindent
{\bf {Case II}} $-$ $\ft$ does not fix any leaf of $\widetilde \Lambda^{su}$.

The analysis of this case will be divided in two subcases that will be proven in section \ref{s.TR}.  The fact that it can be divided into these subcases needs $M$ to be hyperbolic so that Theorem \ref{teo.clasif} applies. Even if the dichotomy holds (without assuming that $M$ is hyperbolic), we need $M$ to be hyperbolic to treat one of the cases (the double translation).

Up to a finite cover and iterate, we may assume that $f$ leaves
invariant branching foliations $\fbs, \fbu$.
By Theorem \ref{teo.clasif}, one considers the action of $\ft$ on
the leaf spaces of $\wfbs, \wfbu$: either $f$ is a discretized Anosov flow ($\ft$ fixes
every leaf of both $\wfbs, \wfbu$), or $\ft$ acts as a translation
on both leaf spaces, called {\em double translation}. 
Recall that $\ft$ does not fix a leaf of $\widetilde \Lambda^{su}$ if
and only if $\widetilde h$ acts as a translation on the
leaf space of $\widetilde \fol$.
Notice that in this particular case
there are three foliations: $\fbs, \fbu$ and $\fol$.
There are actions on the leaf spaces of each of these three
foliations.
Here we will
prove:

\begin{proposition}\label{p.transl1}
Suppose that $\pi_1(M)$ is not virtually solvable.
If $\cF$ is $\RR$-covered and uniform and $\widetilde h$ acts as a translation on 
the leaf space of $\widetilde \fol$
then $f$ cannot be a discretized Anosov flow. 
\end{proposition}

As stated, 
this proposition only requires the manifold not to have (virtually) solvable fundamental group (see \cite{Hammerlindl} for a study of this case). The next result however really uses that the manifold is hyperbolic. 

\begin{proposition}\label{p.transl2}
Suppose that $M$ is a hyperbolic 3-manifold.
If $f$ is a double translation on 
$\wfbs, \wfbu$ then $\widetilde h$ cannot act as a translation on 
$\widetilde \fol$.
\end{proposition} 

This treats all possibilities and therefore these propositions complete the proof of Theorem B (and as a consequence proves Theorem A, see subsection \ref{ss.ergodicity}).

\section{Gromov hyperbolicity, and contractible fixed points}

Before we begin the proof of Theorem B, we obtain some 
general results that will be useful later. The main point of this section is to show that leaves of $\Lambda^{su}$ are Gromov hyperbolic so that Candel's uniformization theorem \cite{Candel} applies. We remark that this property is quite direct in hyperbolic 3-manifolds as they are atoroidal (see e.g. \cite[Chapter 7]{Calegari}), so the reader can skip this section if it is only interested in the hyperbolic 3-manifold case. This section also shows Lemma \ref{fmin} which allows to consider finite lifts and remain being non-wandering to apply Theorem \ref{teo.nocontractible}. This is also not needed in the case where $M$ is a hyperbolic 3-manifold where Theorem \ref{teo.nocontractible} holds without any further assumption.  

\subsection{A general result about foliations with transverse invariant measures} 

 The following is a general result about minimal foliations in 3-manifolds. We have not found this statement in the literature, so we give a proof; but it is likely to be known by the experts.  Raul Ures has informed us that he also has a proof of some parts of this result
\cite{Ures}.  We refer the reader to \cite[Chapters 11 and 12 Book I]{CandelConlon} for background on transverse invariant measures and growth of leaves. 
In the proof of the next theorem
we will also use some standard results from foliations theory in dimension 3 that can be found e.g. in \cite[Chapter 9 Book 2]{CandelConlon}. 

\begin{theorem}\label{minimality}
Let $\fol$ be a minimal 
codimension one foliation in a closed $3$-manifold $M$ admitting a holonomy invariant transverse measure $\nu$. Then, the foliation is $\mathbb{R}$-covered and uniform 
 and one of the following options holds for $\cF$:
\begin{itemize}
\item If there is a planar leaf in $\fol$, then $M$ is the 
$3$-dimensional torus;
\item If there is a leaf which is an annulus or M\"{o}bius band, then $M$ is a nil-manifold (which could be the $3$-dimensional torus);
\item Otherwise in the universal cover $\mt$ all the leaves
are uniformly Gromov hyperbolic.
\end{itemize}
If moreover $\cF$ is transversely orientable, then all leaves are pairwise homeomorphic.
\end{theorem}

If the transverse orientability hypothesis is not fulfilled, one can obtain 
it via taking a double cover (which will not affect the existence of a holonomy invariant transverse measure). Minimality of the foliation after 
a double lift is not immediate, we explain how to obtain this along the
way.

We divide the proof into several lemmas.

First, notice that minimality implies that there are no compact leaves, in particular no Reeb components. 
In addition the foliation does not have sphere or projective plane
leaves.
Therefore one can apply Novikov's and Palmeira's (see \cite{CandelConlon}) theorems to show that $\mt$ is diffeomorphic to $\RR^3$ and leaves of $\tilde \fol$ are properly embedded planes in $\mt$. 

Also, thanks to \cite[Lemma 3.3]{CalegariSM} one can without loss of generality assume that leaves of $\fol$ are smoothly immersed and with immersions varying continuously in the $C^\infty$ topology. This is obtained after a global isotopy of the original foliation which does not affect minimality, the existence of a holonomy invariant transverse measure nor the topological type of the foliation. If one picks a Riemannian structure on $M$ one can 
consider the
unit vector field orthogonal to $T\fol$. This vector field is not
necessarily $C^1$, but can be approximated arbitrarilily close by
a smooth vector field. Pick one such smooth vector field,
 and consider the one dimensional foliation $\tau$ obtained by integrating 
this vector field. 
 
 Cover $M$ by finitely many charts of $\fol$ which are
also foliated charts for the one dimensional foliation
$\tau$. In each chart any transverse arc from the
``bottom" leaf of $\fol$ to the ``top" leaf of $\fol$
has exactly the same measure under $\nu$, by the holonomy
invariance. In addition this measure is not zero, because
of the minimality of $\fol$. The support of the measure
has to intersect the interior of this chart.
Hence there is a minimal $a > 0$ measure of transverse arcs
for any of the elements of this finite family of charts
and also a maximum $b > 0$. 

Let $\wcF$, $\widetilde \tau$ the lifts to
$\mt$. We first establish the fact that the foliation is $\RR$-covered. The proof will help to show that all leaves are homeomorphic in a double cover. 

\begin{lemma} \label{rcov}
The foliation $\fol$ is $\RR$-covered and uniform. 
\end{lemma}

\begin{proof}
Start with a leaf $L$ of $\wcF$ and fix a point $x$ in $L$.
Choose a near enough leaf $E$ of $\wcF$ so that the leaf
of $\widetilde \tau$ through $x$ intersects $E$.
For a point $y$ in $\mt$, let $\widetilde \tau(y)$ be the leaf of $\widetilde \tau$
through $y$.
Let 

$$D = \{ y \in L, \ \ {\rm so \ that} \ \ 
\widetilde \tau(y) \cap E \not = \emptyset \};
\ \ \ {\rm let} \  \varphi: D \rightarrow E,
\ \varphi(y) = \widetilde \tau(y) \cap E.$$

\noindent
Then we claim that $D = L$ and the length of the 
$\widetilde \tau$ segments from $L$ to $E$ have bounded
length.
Clearly $D$ is open. Suppose that $y_i$ are in $D$
converging to $y$ in $L$. Consider the segments
$v_i$ of the foliation $\widetilde \tau$ from
$y_i$ to $\varphi(y_i)$. If these segments
have bounded length, then they are all in a compact
set in $\mt$. Then by the local product structure
of the foliation $\tau$ lifted to $\mt$, it follows
that, up to subsequence, the segments converge
to a segment of the foliation $\widetilde \tau$ from
$y$ to $E$. Hence $y$ is in $D$.

Suppose on the other hand that the length of $v_i$ converges
to infinity. By the holonomy invariance of the measure
$\nu$, it follows that they all have the same measure
($\widetilde \nu$, the lift of $\nu$ to $\mt$).
Projecting to $M$ these are segments of $\tau$ of
length going to infinity. At most $b$ length of these 
segments can be contained in one of finitely many foliated boxes. Hence
the sum of the measures of these
segments is going to infinity,
contradiction.

Therefore the length of $v_i$ is bounded. 
It now follows that $D$ is both open and closed, and so
$D = L$. In addition the length of the $\widetilde \tau$
segments from $L$ to $E$ is bounded.
We stress this fact

\vskip .05in

($\ast$) every $\widetilde \tau$ leaf intersecting $L$ also
intersects $E$, and the length of $\widetilde \tau$ segments
from $L$ to $E$ is bounded.
\vskip .05in

This implies that if $W$ is the region of $\mt$ bounded by
$L, E$ (including $L, E$), then the foliation $\wcF$
restricted to $W$ has leaf space homeomorphic to a closed
interval. 

The same holds for images of $W$ under deck transformations,
that is $\gamma(W)$ where $\gamma$ is in $\pi_1(M)$.
Hence given $\gamma, \beta$ in $\pi_1(M)$,
if $\gamma(W)$ and $\beta(W)$ intersect the foliation $\wcF$ 
has leaf space a closed interval in the union of these two sets.
Since the foliation $\cF$ is minimal, it follows that
the union of deck translates of $W$ cover $\mt$.

This shows that $\fol$ is $\R$-covered.

\vskip .1in
We now show that $\fol$ is uniform. Any two leaves $Z,T$ 
of $\widetilde \cF$ in
$W$ are a finite, bounded distance from each other because
of property $(*)$ above. But we just proved that
the $\pi_1(M)$ translates of $W$ cover $\mt$.
Therefore for any $Z, T$ leaves of $\widetilde \cF$, they
are a finite Hausdorff distance from each other, that is,
$\cF$ is uniform.
This finishes the proof.
\end{proof} 

The property ($\ast$) obtained in the proof is important to obtain:

\begin{lemma} Suppose that $\fol$ is transversely orientable.
Then the leaves of $\fol$ are pairwise homeomorphic.
\end{lemma} 

Here we shall use the transverse orientability of $\fol$. The possible necessity of this is demonstrated by
the following foliation: start with the product foliation
of $S^2 \times S^1$ with transverse measure given by the
$S^1$ measure. Let $\eta$ be a free involution of $S^2$ and
take the quotient of $S^2 \times S^1$ by the involution
$\eta'(p,t) = (\eta(p),1-t)$ there $t$ is mod one.
The quotient has a foliation with spheres and two
projective planes, and a holonomy invariant measure. Not
all leaves are homeomorphic to each other. Here the foliation
is not minimal. We do not know whether this behavior can
occur with minimal foliations. If one does not have transverse orientability, one can always lift to a double cover to get it (minimality of the lifted foliation can be tricky, we explain this at the end). 

\begin{proof}
Let $L, E$ as in the proof of the previous lemma.
Let $\gamma$ in the stabilizer of $L$, in other words,
$\gamma$ is in $\pi_1(\pi(L))$.
Fix a basepoint $x$ in $L$. Consider the $\widetilde \tau$
segment $v$ from $x$ to $z = \varphi(x)$. Choose a path $\alpha$
in $L$ from $x$ to $\gamma(x)$. Pushing this path along
the $\widetilde \tau$ foliation, this produces a path
$\alpha_L$ from $z$ to another point $w$ in $E$.
We can push the whole path because of fact ($\ast$) above.
On the other hand,
the image of $v$ under $\gamma$ is a segment of 
$\widetilde \tau$, starting in $\gamma(x)$ and with
same $\widetilde \nu$ length as $v$. But the $\widetilde
\tau$ segment from $\gamma(x)$ to $E$ also has this
same length. Since no non degenerate segment in $\widetilde
\nu$ has zero length, because of the holonomy invariance
of $\widetilde \nu$; \ it follows that $\gamma(v)$ has to
end in $E$. In other words 

$$\gamma \varphi(x) \ \ = \ \ \varphi(\gamma(x)).$$

\noindent
In fact this works for any $x$ in $L$. It now follows that
$\pi(L)$ is homeomorphic to $\pi(E)$, and this is true for
any leaf of $\wcF$ in between $L$ and $E$. Since the leaves
of $\fol$ are dense the result follows.
\end{proof}


We first give a completely general result about lift
of minimality to finite covers. 
In this result there is no restriction on dimension
of the manifold or codimension of the foliation. The proof that follows was communicated to us by Andy Hammerlindl. 

\begin{lemma}\label{double}
Let $\cH$ be a minimal foliation in a compact
 manifold $P$
and $\cK$ the lift of $\cH$ to a finite cover $N$
of $P$. Suppose that $N$ is connected.
Then $\cK$ is a minimal foliation.
\end{lemma}

\begin{proof}{}
We may assume without loss of generality that the cover 
is regular. Let $\Gamma$ be the group of deck transformations
of the cover $\pi_N: N \rightarrow P$.

Let $\Lambda$ be a non-empty minimal 
sublamination $\Lambda$ of $\cK$ on $N$.
The goal is to show that $\Lambda$ is all of $N$.
Minimality of the foliation $\cH$ on $P$ implies that

$$\bigcup_{\gamma \in \Gamma} \  \gamma(\Lambda)$$

\noindent
is dense in $N$.
Therefore for any $x$ in $N$ there is a sequence of points 
$x_i$ in  $\gamma_i(\Lambda)$ converging to $x$.
Since $\Gamma$ is finite we 
take a subsequence, and assume that $(\gamma_i)$ is
a constant sequence in $\Gamma$.
As  $\gamma_i(\Lambda)$ is closed, this implies that 
$x$ is in $\gamma_i(\Lambda)$.

We have shown that 

$$\bigcup_{\gamma \in \Gamma} \gamma(\Lambda) \ \ = \ \ N.$$

\noindent
By minimality of $\Lambda$, if $\gamma^i(\Lambda)$
intersects $\gamma^j(\Lambda)$,
then the two sets are equal. 
Hence, $N$ is a disjoint union of closed sets of the
form
$\gamma_i(\Lambda)$. As $N$ is connected, 
there can be at most one such closed set and so
$\Lambda = N$, so $\cK$ is minimal.

This finishes the proof.
\end{proof}

Now we can prove the theorem.

\begin{proof}[Proof of Theorem \ref{minimality}]

\vskip .1in
We already proved that $\cF$ is $\R$-covered and uniform
in Lemma \ref{rcov}.
Now we prove the trichotomy in the statement of Theorem
\ref{minimality}.
Lift to a double cover $M_2$ so that the lift $\cG$ of
$\cF$ is transversely orientable. By the previous lemma,
$\cG$ is minimal, and also the leaves of $\cG$ are 
pairwise homeomorphic.

\vskip .08in
\noindent
{\bf {Case 1}} $-$ Suppose that $\cG$ has a plane leaf.

Then all the leaves of $\cG$ are planes. 
Then leaves of $\cG$ cover those of $\cF$ at most two to
one, so the fundamental group of leaves of $\cF$ is a subgroup
of ${\bf Z}_2$. If the fundamental group is not trivial,
then the leaf is the projective plane, which is disallowed
by minimality of $\cF$. It follows that all the leaves
of $\cF$ are also planes. 
Then $M = \mathbb{T}^3$,
by a result of Rosenberg \cite{Rosenberg}. 

\vskip .08in
\noindent
{\bf {Case 2}} $-$ Suppose that $\cG$ has an annulus or
a M\"{o}bius band leaf.

We first rule out a M\"{o}bius band leaf of $\cG$.
Suppose that $E$ is a M\"{o}bius band leaf. Let $\alpha$ be a 
simple closed curve in $R$ so that it has a small neighborhood
$U$
which is homeomorphic to a compact M\"{o}bius band.
Then $E - U$ is an open annulus.
Now take a small transversal neighborhood of $U$ in $M$
foliated by the transversal foliation. Since there is
a holonomy invariant transverse measure, this neighborhood
is product foliated by $\cG$. In particular all the local
leaves of $\cG$ in this neighborhood are compact M\"{o}bius
bands. 
There are infinitely many returns of $E - U$
to the fixed transversal neighborhood of $U$, because $E$ is
dense. The local leaves of $E$ intersected with this
neighborhood are 
local leaves of $\cG$ which are compact M\"{o}bius bands. This
contradicts that $E - U$ is an open annulus.

This shows that there are no M\"{o}bius band leaves of $\cG$
and all leaves of $\cG$ are annuli.

Let $V$ be a leaf of $\cG$. 
Fix a simple not null homotopic closed curve $\delta$ in $V$ through
a point  $x$.
Then $\delta$ generates the fundamental group of $V$.
Since $\nu$ is holonomy invariant and $V$ is dense, the 
curve $\delta$ lifts to a nearby closed curve $\beta$ through
$y$ in $V$. In addition $\beta \cup \delta$
is the boundary of an embedded
annulus $A$ in $M_2$ made up of very small segments of the fixed transverse
foliation. In addition since $V$ is an annulus, then $\beta \cup \delta$
also bound a unique annulus $B$ contained in $V$. We can cut and
paste, choosing the first such intersection of $B$ with the
interior of $A$, so that $B$ does not intersect the interior
of $A$. For any $\epsilon > 0$,
we can also choose $B$ very big in $V$ so that $B$
is $\epsilon$ dense in $M$. In particular we can choose $B$ so
that it intersects every flow line of the transverse flow.

The union $A \cup B$ is a torus, transverse to the flow
along $B$. We can slightly adjust it along $A$ so that it
is transverse to the flow. The resulting torus $T$ is transverse
to the flow and it intersects every orbit of the transverse
flow. In other words the torus is a cross section of the
flow. Hence the manifold fibers over the circle with
fiber a torus.  

\vskip .08in
Suppose first that $M$ is a solv manifold, but is not a
nil manifold.
It is proved in appendix B of \cite{HP-Nil} that $\cG$ is
weakly equivalent to either a stable or unstable foliation
of a suspension Anosov flow or a fibration by tori.
In the first case $\cG$ has to have planar leaves,
in the second case $\cG$ has to have tori leaves. 
Any of these is disallowed by the conditions here.
This implies that $M$ is a nil manifold,
which could be $\mathbb T^3$.

This finishes the analysis of Case 2.

\vskip .08in
\noindent
{\bf {Case 3}} $-$ No leaves of $\cG$ are planes, annuli
or M\"{o}bius bands.

In particular since there are no compact leaves,
it follows that the leaves of $\cG$ cannot be conformally
elliptic or parabolic and they are all conformally hyperbolic.
Each one is separately uniformized with a metric of
constant sectional curvature $-1$.

In this case we use the results of Candel in
\cite{Candel}. In section 4.2 of \cite{Candel} he proves
that if all leaves of $\cG$ uniformize to being hyperbolic,
then the uniformization is continuous and one can
choose a metric in $M$ so that each leaf of $\cG$ has
a metric of constant negative curvature.

It now follows that the leaves of $\widetilde \cG$ 
(which are the same as the leaves of $\widetilde \fol$)
 are uniformly Gromov
hyperbolic in any metric. This proves Case 3.

This finishes the proof of Theorem \ref{minimality}.
\end{proof}

\subsection{The partially hyperbolic case} 

We now apply the general result of the previous subsection to the partially hyperbolic case. The goal is to get Corollary \ref{cor.gromov}, Gromov
hyperbolicity of leaves. We point out that this result is easy in the case where $M$ is a hyperbolic 3-manifold, but we deal with all cases at once here. 

For this we will need to understand the minimal sublaminations of $\Lambda^{su}$, for this, it is useful to use the octopus decomposition of a closed lamination in a foliation (see \cite[Proposition 5.2.14]{CandelConlon}). We state here the main properties we will use: let $\cF$ be a foliation of a closed 3-manifold without compact leaves and consider a closed sublamination $\Gamma$ (i.e. a closed non-empty set saturated by leaves of $\cF$). We can assume that $\Gamma \neq M$. 

Let $U$ be the metric completion of a complementary region
of $\Gamma$. Recall that we can think of the interior $\mathring{U}$
of $U$ as a subset of $M$.
Hence $\mathring{U}$ is foliated by leaves of $\fol$ and by the
completion operation we can think of $U$
also as saturated by leaves of $\fol$.
The set  $U$ has an octopus decomposition \cite[Proposition 5.2.14]{CandelConlon}. This means the following:

$$U \ \  = \ \  K \cup D,$$

\noindent
where $K$ is compact, $D$ is non compact
and $D$ is an $I$-bundle over a non compact surface.
In addition the $I$-bundle fibration in $D$ can be chosen
transverse to the foliation $\fol$ in $D$.
The set $K$ is called a {\emph{core}} for $U$. 
The components of $D$ are called the {\emph{arms}}\footnote{The terminology octopus, core, arms is standard in
foliation theory \cite{CandelConlon}.}.

An arm $V$ is an $I$-bundle over a non compact, connected
surface $S$ with boundary. As one escapes compact sets
in $S$, the corresponding $I$-fibers in $V$ have 
length converging to zero. Using this we can show:

\begin{lemma}\label{octopus}
Let $\fol$ be a foliation without compact leaves and $\Gamma$ be a compact sublamination. Then, if $L$ is a leaf of $\fol$ at positive distance from $\Gamma$, it follows that $L$ is contained in the core of the octopus decomposition of some complementary region of $\Gamma$. 
\end{lemma}

\begin{proof}
Since the foliation is transverse to the $I$-bundle structure in the arms, it follows that if the leaf $L$ intersects some arm, then it must cover the surface in the base of the $I$-bundle and intersect all $I$-fibers in the
arm. Thus the distance from $L$ to the boundary goes to $0$ going
arbitrarily deep in the arms.
This contradicts that $L$ is at positive distance to $\Gamma$. 
\end{proof}

We will now need the following result which will use the dynamics and not only the geometry of the lamination. 

\begin{lemma}\label{minimal}
Suppose that $f$ is a partially hyperbolic diffeomorphism
in $M^s$ with $f$ not accessible and such that the non-wandering set of $f$ is all of $M$. Suppose further that there are no compact surfaces tangent to $E^s \oplus E^u$.
Then the lamination $\Lambda^{su}$ contains a unique
minimal sublamination $\Lambda$.
In particular $\Lambda$ is $f$-invariant.
\end{lemma}

\begin{proof}{}
Let $\Lambda$
be a minimal sublamination of $\Lambda^{su}$. We want to show that $\Lambda$ is $f$-invariant. If we show this, we can apply Theorem \ref{t.HHU} (recall Remark  \ref{r.HHU2}) to $\Lambda$ and get that  
the completion of the  complementary regions of $\Lambda$
are $I$-bundles. If $L$ is any leaf of $\Lambda^{su}$ which is not
in $\Lambda$, then $L$ is contained in one of these $I$-bundles.
As in Lemma \ref{octopus} one can prove  that the closure of $L$ 
has to contain $\Lambda$ and this shows that $\Lambda$ is the unique minimal sublamination of $\Lambda^{su}$.

We will therefore consider  the sets $f^i(\Lambda), \  i \in \ZZ$, and the lamination 
$$E \ \ = \ \ \overline{\bigcup_{i \in \ZZ} f^i(\Lambda)}.$$ 

\noindent
Notice that $f^i(\Lambda)$ is contained in $\Lambda^{su}$. Since $\Lambda$ is minimal then for any $i, j$
either $f^i(\Lambda)  = f^j(\Lambda)$ or they are disjoint. 
	Then $E \subset \Lambda^{su}$ is a sublamination and $E$ is $f$-invariant.

Because $E$ is $f$-invariant, we can now apply Theorem \ref{t.HHU} (recall Remark \ref{r.HHU2}). It follows that the completion of a complementary region of 
$E$ is an $I$-bundle. 
		
 In the beginning of section \ref{section4} we explained that $\Lambda^{su}$ can be extended to a foliation $\fol$,
which is not necessarily $f$ invariant. Clearly $\fol$ extends also $\Lambda$ and $E$ so we can apply the previous lemma to both.

We start by analyzing some properties of $\Lambda$, which we can assume is not $M$ (otherwise there is nothing to prove). 
We can use the octopus decomposition of 
any complementary region of $\Lambda$, as explained above. Using Lemma \ref{octopus} we see that $f^{i}(\Lambda)$ must be contained in the core of the octopus decomposition of some complementary region of $\Lambda$ for every $i$ such that $f^i(\Lambda)\neq \Lambda$. 

Suppose that $E \not = \Lambda$. Then there is a complementary
region $V$ of $\Lambda$ which intersects $E$ 
in $f^i(\Lambda)$ for some $i \in \mathbb Z$. Let $U$ be the metric
completion of such
a complementary region $V$, with octopus decomposition $K \cup D$ as above.
We know that $E \cap V$ is contained in $K$. In addition 
$E \cap V$ cannot get very close to a boundary leaf of $V$ or else
using the local product structure of $\fol$ and that $K$
is compact it would imply that $E$ intersects the interior
of $D$. 
In particular there is $\eps > 0$ so that if $x$ is in 
$E \cap V$ then $d(x,\Lambda) > \eps$.

\vskip .1in
Now we will use both sublaminations $\Lambda$ and $E$
and appropriate complementary components of these
sublaminations.
First there is the 
complementary region $V$ of $\Lambda$ with completion $U$
as above.

Let $L$ be a boundary leaf of $U$. 
Recall that $E \cap V$ is contained in the interior
of $K$ and $E$ does not intersect the interior of $D$.
In other words there are no leaves of $E$ between
$E \cap V$ and $L$.
It follows that there is a complementary region
$A$ of $E$ which 
contains points of $E \cap V$ and also of $L$ in the boundary
(notice now we considering
$A$ a complementary region of $E$ and not of $\Lambda$).
In particular $L$ is contained in the boundary of $A$.

There is at least one other boundary leaf $F$ of $U$ which
is also contained in the boundary of $A$: starting from
a point in $V$ less than $\eps$ near $L$, move along
in $V$ 
until you hit an arm of $U$, that it, until you intersect
$D$. This can
be done avoiding $E \cap V$ since all points in $E \cap V$
are at least $\eps$ distant from any point in the boundary
of $U$.
This arm of $U$ has two boundary leaves, one is contained
in $L$ and the other is contained in another boundary
component $F$ of $U$.
Since $F$ is in $\Lambda$ it is also in $E$. 
Since the metric completion of $A$ is a foliated
$I$-bundle, there are only two leaves in the boundary
of this metric completion of $A$. This implies that
the boundary of $A$ is $L \cup F$.
This is impossible because $E \cap V$ intersects $A$.


The contradiction was obtained by assuming that $f^i(\Lambda)$
is distinct from $\Lambda$ for some $i$.


The conclusion is that $\Lambda$ is $f$ invariant.
In particular $E = \Lambda$ as we wanted to show.
\end{proof}

%

We stress that this has no assumption on $M$, or homotopic
assumptions on $f$. We only have the dynamical hypothesis (non-wandering of $f$ is all of $M$)
on $f$. Lemma \ref{minimal} allows us to use Theorem \ref{minimality} to 
obtain Gromov hyperbolicity of leaves of $\Lambda^{su}$ using
also the following result:

\begin{proposition}\label{grhyp}
Let $\Lambda \subset M$ be a minimal lamination
 without compact leaves
 and so that the complementary regions of $\Lambda$ 
have completions that are 
are $I$-bundles. Assume that $M$ is an irreducible 3-manifold whose fundamental group is not (virtually) solvable. Then, leaves of $\Lambda$ are uniformly Gromov hyperbolic.
\end{proposition} 

\begin{proof}{}
 Suppose that leaves
of $\Lambda$ are not uniformly Gromov hyperbolic. 
By Candel's theorem \cite{Candel} 
there is a holonomy invariant transverse measure $\nu$ to
$\Lambda$.

The support of $\nu$ is a sublamination of $\Lambda$.
By hypothesis, this lamination is minimal, so
the support of $\nu$ is all of $\Lambda$.

The complementary regions to $\Lambda$ have completions that
are $I$-bundles. There are two 
possibilities. Suppose that $\Lambda$ has a compact leaf.
By minimality $\Lambda$ is a single leaf $G$.
Then up to a double cover $G$ is a fiber of a fibration
of $M$ over $S^1$.
If $G$ is a sphere, projective plane, torus or Klein
bottle, then $\pi_1(M)$ is virtually solvable, contrary
to hypothesis. Hence $G$ is Gromov hyperbolic and
the result is proved.

The other option is that $\Lambda$ does not have compact leaves.
Hence one can blow down these complementary
regions so that $\Lambda$ blows down to a foliation
$\hol$. Since $\Lambda$ is minimal, then $\hol$ is
minimal. The holonomy invariant transverse measure $\mu$
blows down to a holonomy invariant transverse measure
to $\hol$. The support is all of $M$. Now one can apply Theorem \ref{minimality} and the fact that $\pi_1(M)$ is not (virtually) solvable to conclude. 
This finishes the proof.

We remark that 
if $\Lambda$ was a single compact leaf $G$ we could not do
the blow down procedure, as it would blow down $M$ to the
single compact surface $G$.
\end{proof}

This allows us to get the following result which is what we will need. 

\begin{corollary}\label{cor.gromov}
Let $f$ be a partially hyperbolic diffeomorphism of $M$ 
which is not accessible and so that the non-wandering
set of $f$ is $M$.
Suppose that $\pi_1(M)$ is not virtually solvable.
Then the leaves of $\Lambda^{su}$ are uniformly Gromov hyperbolic. 
\end{corollary}

\begin{proof}{}
By Lemma \ref{minimal}, the lamination $\Lambda^{su}$ has a unique
minimal sublamination $\Lambda$ and it is $f$-invariant.
By Theorem \ref{t.HHU} the completion of the complementary regions 
of $\Lambda$ are $I$-bundles.
Hence by Proposition \ref{grhyp} the leaves of $\Lambda$ are
Gromov hyperbolic. 
Using again that the completion of the complementary regions of
$\Lambda$ are $I$-bundles, it follows that 
the leaves of $\Lambda^{su}$ in the 
complement of $\Lambda$ are also uniformly Gromov hyperbolic.
\end{proof}

\subsection{Finite lifts of iterates and contractible periodic points}

We will also need the following two results.
The first is a completely general result.

\begin{lemma}\label{fmin2}
Let $f$ be a homeomorphism of $M$ which is non-wandering.
Let $g$ be a lift of an iterate of $f$ to a finite
cover $M_1$ of $M$. Then $g$ is non-wandering as well.
\end{lemma}

\begin{proof} 
Let $g: M_1 \rightarrow M_1$ be a finite lift of an iterate of $f$.
Up to taking a further cover and another iterate 
we can assume that 
$M_1$ is a normal cover.

Since $\Omega(f) = M$, the set of recurrent points of $f$ is dense
in $M$. This is because for a basis of the topology $\{U_n\}_n$ we have that the sets

$$A_n \  = \ \{ x \in U_n \ : \ \exists k>1 \ , \ f^k(x) \in U_n \} 
\ \cup \ \overline{U_n}^c$$

\noindent
 are open and dense and a point in $\bigcap_n A_n$ must be recurrent. 

Also recurrent points are non-wandering points for iterates of $f$:
a recurrent point $x$ is in the omega limit set $\omega(x,f)$ of $x$ with respect to $f$. It is easy to see that

$$\omega(x,f) \ \ = \ \ \omega(x,f^n) \cup \omega(f(x),f^n) \cup \cdots
\cup \omega(f^{n-1}(x), f^n).$$

\noindent
Hence $x$ is in $\omega(f^i(x),f^n)$ for some $0 \leq i < n$.
Finally since $x$ is in the omega limit set of some point
for $f^n$, then $x$ is in the non-wandering set of $f^n$.
This implies that the result is true for iterates of $f$.

So we can assume that $g$ is a lift of $f$ and not of an iterate. 
Let $x$ be the lift to $M_1$ of a
recurrent point. We want to show that $x$ is non-wandering for $g$.
Since the non-wandering set is closed and the lifts of recurrent points of $f$ is dense we would then get
$\Omega(g) = M_1$. 

The cover $M_1 \to M$ is normal, so the group $V$
of deck transformations acts transitively on lifts of points.
Let $\gamma$ be one such deck transformation.
As $\gamma(x)$ is the lift of a recurrent point, there are $n_j \to \infty$
so that $g^{n_j}  \gamma(x) \to \beta(x)$, where $\beta$ is also some deck
transformation of $M_1 \to M$. 
Up to taking a subsequence of $(n_j)$ (still denoted by 
$(n_j)$) we may assume that for
any $\alpha$ in $V$ there is $\delta$ in $V$ so that
$g^{n_j} \alpha(x) \rightarrow \delta(x)$. 
This defines a map $\eta: V \to V$ by $\eta(\alpha) = \delta$.
The important thing to notice that $\eta$ is a bijection of a finite set.
Hence there is a minimal $k > 0$ so that $\eta^k(id) = id$.

If $\beta = \eta^{-1}(id)$ then there is a subsequence $n_j$ 
as above so that $g^{n_j}(\beta(x)) \to \eta(\beta)(x) = x$.
Now take a neighborhood $U$ of $x$, Let $U_{k-1}$ be an open
set around $\eta^{-1}(id)(x) = \beta(x)$ so that for 
some large $m_{k-1}$ one
has that

$$g^{m_{k-1}} (U_{k-1}) \ \subset \ U$$

\noindent
This exists because $g^{n_j}(\eta^{-1}(id)(x))$ converges to $x$.

Similarly, inductively construct $U_i$ neighborhood of $\eta^{-i}(id)(x)$
and $m_i$ so that $g^{m_i}(U_i) \subset U_{i+1}$. Once one
has constructed $U_1$, it follows that for some $m_0$ one
has that $g^{m_0}(x)$ is in $U_1$ and then
$$g^{m_0 + m_1 + \ldots m_{k-1}} (x) \ \in \ U.$$
\noindent
This shows in fact that $x$ is recurrent for $g$,
and hence non-wandering as well.
This shows that if $\Omega(f) = M$, then $g$ is non wandering.
\end{proof}

\begin{lemma}\label{fmin}
Let $f: M\to M$ be a partially hyperbolic diffeomorphism homotopic to the identity and such that the non-wandering set of $f$ is all of $M$.
Then $f$ has no contractible periodic points.
\end{lemma}

\begin{proof} Up to taking a finite lift $g$ of an iterate of $f$ we can assume that the bundles of $g$ are orientable and its orientation is preserved by $Df$. This allows us to apply Theorem \ref{teo-BI} to get branching foliations $\fbs,\fbu$, for the map $g$. 

Assume, say that $\fbs$ is not minimal. Then a minimal $g$-invariant set $\Lambda$ of $g$ is a proper repeller (because it is transversally unstable). This implies that $g$ cannot be non-wandering (see \cite[Section 1.1.2]{CP}).

By the previous lemma, this forces $f$ also to be non-wandering,
contrary to hypothesis. Hence $g$ is non-wandering.

Now we can apply Theorem \ref{teo.nocontractible} to $g$ which implies that $f$ cannot have contractible periodic points either, since one such point would give rise to a contractible periodic point for $g$. 
This finishes the proof.
\end{proof}

\section{Case I $-$ Fixed leaves}\label{s.FL}
This section will be devoted to the proof of Proposition \ref{p.fixedleaf}.

\vskip .1in
\noindent
{\bf {Assumption in section \ref{s.FL}}} $-$
The hypothesis in this section are those of Proposition \ref{p.fixedleaf}. In particular, no need to assume that $M$ is hyperbolic. We will assume that there is a leaf $\tilde \Lambda^{su}$ which is fixed by a good lift $\ft$ of $f$ to $\mt$. 
\vskip .06in

We start by showing that leaves of $\Lambda^{su}$ whose
lifts are fixed by $\ft$ all have cyclic fundamental group,
and there is at least one which is not a plane. 

Suppose that $C$ is a leaf of $\Lambda^{su}$ which has a lift
$L$ to $\mt$ which is fixed by $\ft$. 
The goal to obtain Proposition \ref{p.fixedleaf} is to show
that this assumption on the existence of $C$ leads
to a contradiction.
Since $\ft$ commutes
with deck transformations, then $\ft$ fixes any lift of $C$ to
$\mt$. Hence this is a property of $C$ and not of the particular lift.

The proof of the following lemma will make use of the theory of \emph{axes} for free actions on leaf spaces of foliations, 
we refer the reader to \cite{Fen2003} or \cite{Ro-St}
 for a general account or \cite[Appendix E]{BFFP} for a more direct account for use in a similar context. 

\begin{lemma}\label{l.cyclicstab}
Let $L \in \widetilde \Lambda^{su}$ which is fixed by $\ft$. Then, $\mathrm{Stab}_L :=\{ \gamma \in \pi_1(M)  \ : \ \gamma L = L\}$ is cyclic. 
\end{lemma}

\begin{proof}
Suppose that $L$ is fixed by $\ft$. In this case we proceed 
as in the proof of \cite[Proposition 3.14]{BFFP}:  
By Lemma \ref{fmin} the map $\ft$ has no periodic points.
It follows that $\ft$ does
not fix any stable leaf in $L$: otherwise if $s$ is such a leaf,
since $\ft$ is a contraction in the stable leaves,
there would an $\ft$ periodic point in $s$.
Hence $\ft$ acts freely on the stable leaf space of $L$, i.e. $\cL^s_L = L/_{\widetilde \cW^s}$. Therefore, there is an axis $\mathrm{Ax}^s$ of $\ft$ acting on $\cL^s_L$. 
Similarly $\ft$ acts freely on the unstable leaf space
$\cL^u_L$ and has an axis $\mathrm{Ax}^u$.

The axis is the set of leaves $s$ so that $\ft(s)$ separates
$s$ from $\ft^2(s)$ in $L$. In \cite{Fen2003} it is shown
that if $\ft$ acts freely on the leaf space, then 
the axis is non empty, and the axis 
 is either a line or a $\ZZ$-union of disjoint intervals in 
$\cL^s_L$ (see also \cite[Appendix E]{BFFP}). In the second case 
$$\mathrm{Ax}^s \ = \ \bigcup_{i \in {\ZZ}}
I_i \ = \ \bigcup_{i \in {\ZZ}} [x_i,y_i]$$

\noindent
Here $[x_i,y_i]$ are closed intervals in the leaf space
$\cL^s_L$ and $y_i$ is non separated from $x_{i+1}$ in 
$\cL^s_L$.

In addition, there are no closed stable leaves in $M$, hence any deck transformation fixing $L$ must act freely on $\cL^s_L$ too. As deck transformations commute with $\ft$ it follows that any deck transformation fixing $L$ must have the same axis as $\ft$ (see \cite[Lemma 3.11]{Fen2003}).

This implies that deck transformations all act without fixed points in the single set
$\mathrm{Ax}^s$. If $\mathrm{Ax}^s$ is a line, then
by H\"{o}lder's theorem the group must be abelian (see \cite[Proposition E.2]{BFFP}). In other words,
the fundamental group of the leaf must be abelian. 
If the axis is an infinite union of intervals, then the fundamental
group of $\pi(L)$ acts on this collection without any
fixed points, that is $\pi_1(\pi(L))$ acts freely on $\ZZ$.
Again this implies that $\pi_1(\pi(L))$ is abelian.
Since there are no closed leaves of $\cF$ one obtains that $\mathrm{Stab}_L$ is cyclic. 
\end{proof}

\begin{lemma}\label{l.onenotplane}
Suppose that $\ft$ fixes a leaf of $\widetilde \Lambda^{su}$.
Then there exists at least one leaf $L$ of $\widetilde \Lambda^{su}$ which is fixed by $\ft$ and has stabilizer $\mathrm{Stab}_L :=\{ \gamma \in \pi_1(M)  \ : \ \gamma L = L\}$ which is isomorphic to $\ZZ$. 
\end{lemma}

\begin{proof}
Assume by way of contradiction 
that all leaves $C$ of $\Lambda^{su}$ which have
a lift $L$ to $\mt$ fixed by $\ft$ are planes (in other words
if $L$ is such a lift, then $\mathrm{Stab}_L$ is trivial, or 
equivalently
$\pi_1(C)$ is trivial). 
Since only $\T^3$ admits a foliation by planes (see \cite{Rosenberg}), one has to have one leaf $C$  of $\cF$ with 
$\pi_1(C)$ not trivial, and 
$\pi_1(C)$ contains $\gamma \in \pi_1(M) - id$.
Since complementary regions to $\Lambda^{su}$ are $I$-bundles, it follows that there is a leaf $B \in \Lambda^{su}$, so that $\gamma$ is in $\pi_1(B)$.
Lift $B$ to a leaf $\hat B$ of $\widetilde \Lambda^{su}$, so that
$\hat B$ is fixed by $\gamma$.
By assumption in the beginning of this paragraph,
$\hat B$ cannot be fixed by $\ft$. By Proposition \ref{prop.dich} 
the completion $U$ of the region between $\hat B$ and $\ft(\hat B)$ is an $I$-bundle.
In addition since $\ft$ has fixed leaves, then
$V= \bigcup_{n\in\ZZ} \ft^n(U)$ is not all of $\mt$, and all leaves
in $\partial V$  are fixed by $\ft$.  
Here $\hat B$ is in $\widetilde \Lambda^{su}$, and each leaf
in $\partial V$
is accumulated by $\ft^n(\hat B)$ with $n \rightarrow \infty$ or
$n \rightarrow -\infty$.  Hence any leaf in $\partial V$ is
in $\widetilde \Lambda^{su}$. Since each such leaf $L$ is fixed
by $\ft$, then $\pi(L)$ it must be a plane, by the assumption
in the beginning of this paragraph.

Again, by Proposition \ref{prop.dich} the set $V$ is precisely
invariant, hence $\pi(V)$ is an open, $\fol$ saturated set,
which is not all of $M$ and whose leaves in the boundary
are all in $\Lambda^{su}$ and are all planes.
One can do the 
octopus decomposition of the completion $W$
of $\pi(V)$ (see \cite[Proposition 5.2.14]{CandelConlon}). 
The decomposition is $W = K \cup D$ (not unique), where
$K$ is compact, and $D$ is an $I$-bundle and very thin (meaning that local product structure boxes are not completely contained in $D$ so that center curves go from side to side).
It follows from the fact that boundary leaves of $\pi(V)$ are planes, that $\pi(V)$ has to be an $I$-bundle, that is,
a  disk times an open interval.
(This uses that $M$ is irreducible.)

But $\pi(V)$ contains $\pi(\hat B)$ which does not have trivial fundamental
group, contradiction.
This completes the proof of the lemma.
\end{proof}
\begin{proof}[Proof of Proposition \ref{p.fixedleaf}.]
Suppose that there is a leaf of $\widetilde \Lambda^{su}$
that is fixed by $\ft$.
We proceed as in \cite[Section 4]{BFFP} to get a contradiction.
By the previous lemma, there is 
a leaf $L$ of $\widetilde \Lambda^{su}$, 
which is fixed by both $\ft$ and $\gamma \in \pi_1(M) \setminus \{ \mathrm{id} \}$.
We will work in $L$.
We just sketch the main arguments and refer the reader to \cite{BFFP} for full details. We remark again that several arguments are simpler in this setting because we know precisely how the dynamics of both foliations look like.

The leaf $L$ has stable and unstable one dimensional foliations, on
which both $\ft$ and $\gamma$ act.

\begin{itemize}
\item As shown in the proof of Lemma \ref{l.cyclicstab}, the maps $\ft$ and $\gamma$ act freely on both $\cL^s_L$ and $\cL^u_L$.
Since $\ft$ and $\gamma$ commute they both share the same
axis $\mathrm{Ax}^s$ in $\cL^s_L$,
and $\mathrm{Ax}^u$ in $\cL^u_L$.
Each of these axes is either a line or a $\ZZ$-union of intervals. 

\item An unstable leaf in $L$ 
cannot intersect a stable leaf $s$ in $\mathrm{Ax}^s$ 
and its image $\gamma(s)$. Otherwise a graph transform argument (as in
e.g. \cite[Appendix H]{BFFP}) would 
produce a closed unstable leaf in $\pi(L)$, which is impossible. 

\item This allows us to find a stable leaf $s_1$ which is fixed by 
$g := \gamma^n \circ \ft^m$ where $m > 0$, see  \cite[Lemma 4.6]{BFFP}. 
This map $g$ is also partially hyperbolic and with uniform constants. 
It also follows that $s_2=\gamma^{-n}s_1 = \ft^m s_1$ is at bounded distance from $s_1$ in $L$,
because $\ft$ moves points a bounded distance inside $L$ (Remark \ref{r.boundedmovement}). 

\item The next step is to show that leaves of $\widetilde \cF$ are 
Gromov hyperbolic\footnote{Here there is a substantial difference with
\cite{BFFP}. There we prove Gromov hyperbolicity for the
leaves of $\fbs, \fbu$ using partial hyperbolicity that gives information on the dynamics transverse to the foliation. In our situation, transverse to the $\Lambda^{su}$ lamination
is the center direction. A priori this could be contracting, expanding
or neither under $f$. Hence this step necessitates a different proof
from what is done in \cite{BFFP}.}.This was done in Corollary \ref{cor.gromov}. Then by Candel's result \cite{Candel},
$M$ admits a metric that makes all leaves of $\cF$ of constant negative curvature.

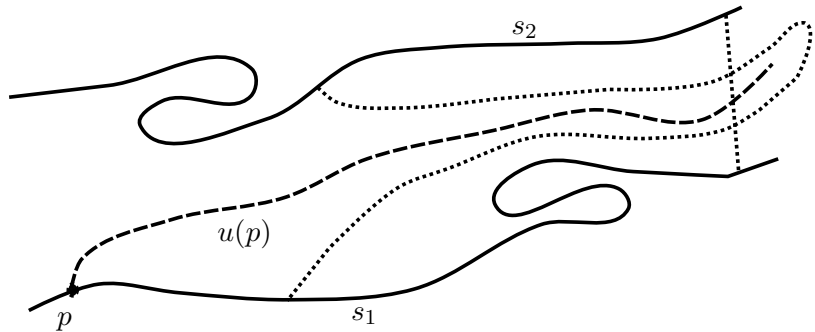
\begin{figure}[ht]
\begin{center}
\includegraphics[scale=1.70]{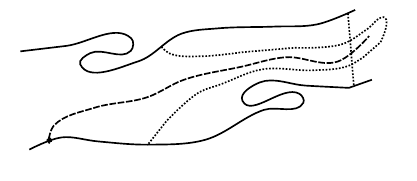}
\begin{picture}(0,0)
\put(-310,20){$p$}
\put(-140,130){$s_2$}
\put(-200,22){$s_1$}
\put(-250,52){$u(p)$}
\end{picture}
\end{center}
\vspace{-0.5cm}
\caption{{\small Such a configuration gives a contradiction with the fact that there are points in $s_1 \setminus \{p\}$ which are mapped very far away, because $g$ cannot map curves of bounded length to arbitrarily long curves.}\label{f.case1}}
\end{figure}

\item Using a comparison with hyperbolic isometries, it  is possible to show that there are points in $s_1$ which are mapped arbitrarily far apart from themselves by $g$ in each ray of $s_1 \setminus \{p\}$ (where $p$ is the fixed point in $s_1$ for $g$). 
This is  \cite[Lemma 4.11]{BFFP}. 

\item One can find an unstable leaf $u_1$ in $L$, fixed by $g$ and
separating $s_1$ from $s_2$ in $L$ (see \cite[Lemma 4.3]{BFFP}).
Using the previous remark and the fact that the band between $s_1$ and
$s_2$ in $L$ has bounded width, this 
 forces the unstable $u_1$ to be coarsely contracting under $g$ like in  \cite[Lemma 4.15]{BFFP} (see figure \ref{f.case1}).
This is a contradiction because $m > 0$.
\end{itemize}
This finishes the proof of Proposition \ref{p.fixedleaf}. 
\end{proof}

\section{Case II $-$ Translation case on
$\widetilde \fol$}\label{s.TR} 

Here we assume that $\ft$ does not fix any leaf of 
$\widetilde \Lambda^{su}$, or equivalently that
$\widetilde h$ does not fix any leaf of $\widetilde \fol$.

\vskip .1in
\noindent

We apply Proposition \ref{prop.dich} to $\fol$ and
$h$, and we obtain that $\fol$ is $\R$-covered and
uniform and that $\widetilde h$ acts as a translation, up to changing the parametrization of the leaf space, we can assume that $\widetilde h$ is an increasing homeomorphism. 

\subsection{Discretized Anosov flow case}\label{ss.DAF}
In this subsection we prove Proposition \ref{p.transl1}. 
The goal is to prove that $f$ cannot be a discretized Anosov flow.
The proof is done under more general assumptions:

\vskip .1in
\noindent
{\bf {Assumption in subsection \ref{ss.DAF}}} $-$ We assume
that $\pi_1(M)$ is not virtually solvable, and therefore, there are no
tori tangent to $E^s \oplus E^u$ (c.f. Remark \ref{rem.toriimplysolv}).


\vskip .06in
Assume by way of contradiction 
that $f$ is a discretized Anosov flow. In particular
$f$ is dynamically coherent 
\cite[Proposition G.2]{BFFP2}, and there exists a (topological) 
Anosov flow $\Phi$ and a positive continuous function $\tau: M \to \RR_{>0}$ so that $f(x)= \Phi_{\tau(x)}(x)$ and center leaves of $f$ are the orbits of $\Phi$. 
In particular since $M$ is compact, there are $a, b > 0$ so that
$a < \tau(x) < b$. Notice that one could apply Theorem \ref{teo.AnosovFlows} if one showed that the flow $\Phi_t$ is regulating for 
$\cF$. This can be done, but we chose to give a more direct proof (which uses similar ideas to the proof of Theorem \ref{teo.AnosovFlows}).  

Since $\Phi_t$ is a (topological)  
Anosov flow, \cite[Corollary E]{Fen1998} shows
that either $\Phi_t$ is topologically conjugate to a suspension
Anosov flow or there are periodic orbits which are freely
homotopic. In fact in the second case there are periodic
orbits $\eta_1$ and $\eta_2$ of $\Phi_t$ which are freely homotopic but with different orientation
(see \cite[Corollary 4.5]{Fen1998}, or \cite[Section 2]{BarthFenley}). See also \cite{BarbotFenley21} and Lemma \ref{lem.orientation} below for the pseudo-Anosov case. 

In other words the following happens:
lift $\eta_i$ coherently to $\widetilde \eta_i$ orbits
of $\widetilde \Phi$. Let $\gamma$ be a deck transformation
associated with $\eta_1$ in the flow forward direction,
hence associated with $\eta_2$ in the backwards direction,
so $\gamma$ preserves both $\widetilde \eta_1$ and
$\widetilde \eta_2$.
Choose points $x_i$ in $\widetilde \eta_i$. Then
\begin{itemize}
\item $\gamma \widetilde x_1 = \widetilde \Phi_{t_1}(\widetilde x_1)$ with $t_1>0$ and,
\item $\gamma \widetilde x_2 = \widetilde \Phi_{t_2}(\widetilde x_2)$ with $t_2 <0$. 
\end{itemize}
Now we will consider $\ft^n(\widetilde x_1)$ as $n \rightarrow \infty$.
Since $f$ is a discretized Anosov flow, then $\ft^n(\widetilde x_1)$
is in $\widetilde \Phi_{\R} (\widetilde x_1)$, and it
is equal to $\widetilde \Phi_{t_n} (\widetilde x_1)$ where
$t_n > an$, since each application of $\ft$ moves at least
$a > 0$ time forward along the orbit. But the flow line $\widetilde \eta_1$ is
a bounded distance from $\widetilde \eta_2$ (because $\eta_1$,
$\eta_2$ are freely homotopic). Therefore 
$\widetilde \Phi_{t_n}(\widetilde x_1)$ is a bounded distance from 
$\widetilde \Phi_{r_n} (\widetilde x_2)$, and here $r_n \rightarrow 
-\infty$ as $n \rightarrow \infty$. This is because $\eta_1$ is
freely homotopic to the inverse of $\eta_2$. But the points
$\ft^i(\widetilde x_2), i < 0$ are evenly spaced along the negative
ray $\widetilde \Phi_{t < 0}(\widetilde x_2)$.
This is because $\tau(x) < b$ for all $x$. It follows
that $\ft^n(\widetilde x_1)$ is a bounded distance from
$\ft^{i_n} (\widetilde x_2)$, where $i_n \rightarrow -\infty$
as $n \rightarrow \infty$.

This will give a contradiction, the idea is that  since $\Phi$ is transverse to $\cF$ orbits must move all in the \emph{same direction with respect to $\cF$}. We carry the details below. 

\vskip .1in
At this point we will only use the maps 
$f, \ft$ and the foliations $\fol$ and $\widetilde \fol$.
First we need more information about the foliation $\fol$.
Here $\fol$ is $\R$-covered and does not have any compact
leaves. One can apply \cite[Proposition 2.6]{Fen2002} and get that
$\fol$ can be collapsed to a minimal
foliation $\cE$ which is minimal 
and which is obtained from $\fol$ by collapsing at most
a countably many foliated $I$-bundles of $\fol$ to
single leaves. This collapsing is homotopic to the identity
and can be lifted to $\mt$: the distances are distorted 
a bounded additive amount.

Now $\cE$ is $\R$-covered, uniform and minimal. 
By \cite[Theorem 2.7]{Th2} one has that $\fol$ is a
{\em slithering} of $M$ around ${\mathbb S}^1$. Slithering
means that there is a fibration $\nu: \mt \rightarrow 
{\mathbb S}^1$
and the deck transformations of $\mt \rightarrow M$ induce
bundle automorphisms. In other words the pre images of $\nu$ form
a foliation in $\mt$ that is $\pi_1(M)$ invariant, and
induces a foliation in $M$, which in this case is $\cE$.
The slithering comes with a coarsely defined distance \cite{Th2}
between leaves $E, E'$ of $\widetilde \cE$:
Parametrize the circle ${\mathbb S}^1$ as $[0,1]$ with $0$ identified
to $1$. The universal cover of the circle is $\R$, where
the generator deck transformation is a translation by
$1$. Now pick one
lift $\widetilde \nu: \nu \rightarrow \R$ of $\nu$.
Given $E, E'$ leaves of $\widetilde \cE$, define
the rough distance $z(E,E')$ as the absolute value
of $\widetilde \nu(E) - \widetilde \nu(E')$. 
In \cite{Th2} Thurston is more careful defining this ``pseudo-distance"
to be an integer, but there is a bounded error between the definitions.
Then \cite[Theorem 2.6]{Th2} shows that the distance between any
two points $x \in E$ and $y \in E'$ is bounded below by a constant
times $z(E,E')$.

Now we are ready to finish the proof of Proposition 
\ref{p.transl1}.
Suppose that $\widetilde x_i$ is in $L_i$, where $L_i$ are
leaves of $\widetilde \fol$.
Then $\ft^n(\widetilde x_1)$ is in
$L_{u_n}$, with $u_n \rightarrow \infty$. We proved 
above that 
$\ft^{i_n}(\widetilde x_2) \in L_{v_n}$, with $v_n \rightarrow -\infty$,
both with $n \rightarrow \infty$.
In addition the distance from $\ft^n(\widetilde x_1)$ to
$\ft^{i_n}(\widetilde x_2)$ is uniformly bounded in $n$.
Since leaves of $\widetilde \cF$ are a uniformly 
bounded distance from leaves in $\widetilde \cE$, the same holds
for the projection of the points to leaves of $\widetilde \cE$.
But the previous paragraph shows that the distance between
any points in $L_{v_n}$ and $L_{u_n}$ is converging to 
infinity because $z(L_{v_n},L_{u_n})$ is converging to
infinity. This is a contradiction.

This shows that this case cannot happen.
This finishes the proof of Proposition \ref{p.transl1}.



\subsection{Pseudo-Anosov flows transverse
to $\R$-covered foliations}

To analyze the case that $\ft$ does not fix
any leaf of $\widetilde \Lambda^{su}$ and is a double translation (in the sense of Theorem \ref{teo.clasif}) we will need some properties
of pseudo-Anosov flows. (This case will be called \emph{triple translation}, cf. subsection \ref{ss.Triple}.) 

Each of the $3$ foliations will come equipped with a transverse pseudo-Anosov flow and to pursue our arguments we need to be able to compare such flows. This is the purpose of this section. We remark that the results proved in this subsection are completely general and can be of independent interest. 

Let  $\cG$ be a foliation or a branching foliation 
(such as $\fbs, \fbu$) in
a hyperbolic $3$-manifold $M$. Let $g$ be a homeomorphism homotopic
to the identity preserving $\cG$. 
If $\cG$ is a branching 
foliation, then we assume that it is approximated 
arbitrarily close by a Reebless foliation, as in the
case of the branching foliations
$\fbs, \fbu$ of a partially hyperbolic diffeomorphism.
As explained in \cite[Section 3]{BFFP2}, one can talk 
about the leaf space of $\wg$. It is the same as the leaf
space of the approximating Reebless foliation.
Suppose that this leaf 
space is $\R$ and that $\cG$ is uniform and transversely orientable.
Here uniform has the same definition as that for regular
(unbranched) foliations.
Then, by Theorem \ref{teo.regul}  there is a 
pseudo-Anosov flow $\Phi_t$ transverse to $\cG$ and
regulating for it (see also \cite[Section 8]{BFFP}).   
This only works if $M$ is hyperbolic.

First we define a very weak form of free homotopy for
periodic orbits of flows.

\begin{definition}{(freely homotopic orbits)} Let 
$\Phi_1, \Phi_2$ be two flows in a closed 3-manifold. We say that
periodic orbits  $\alpha_1$ of $\Phi_1$ and $\alpha_2$ of
$\Phi_2$ are \emph{freely homotopic,} if there are 
$i, j \in \ZZ$, both non zero, 
 so that $\alpha_1^i$ is freely homotopic
to $\alpha_2^j$ as unoriented closed curves.
\end{definition}

In other words,  we do not care about powers or orientations
in this definition.

We will use the following known fact about freely homotopic orbits of pseudo-Anosov flows (see e.g. \cite[Section 2]{BarbotFenley21}) whose proof we indicate for the convenience of the reader familiar with the notion of lozenges.  

\begin{lemma}\label{lem.orientation}
Let $\Phi$ be a pseudo-Anosov flow in a 3-manifold $M$ which has two distinct periodic orbits $\alpha_1, \alpha_2$ which are freely homotopic. Then, there is a pair of orbits $\zeta_1, \zeta_2$ of $\widetilde{\Phi}$ in $\mt$ which are fixed by the same deck transformation $\gamma \in \pi_1(M)$ and such that $\gamma$ acts as a translation on $\zeta_1$ in the same direction as the flow $\widetilde{\Phi}_t$ while it acts as a translation in the opposite direction to the flow on $\zeta_2$. \end{lemma}
\begin{proof}

By \cite[Theorem 4.1]{Fen1998} (which also works for
pseudo-Anosov flows, see e.g. \cite[Proposition 2.11]{BarbotFenley21}) coherent lifts of $\alpha_i$ to 
$\mt$ are connected by a finite chain of lozenges,
and all corners of the lozenges are left invariant
by the same non trivial deck transformation $g$.
We refer to \cite[Section 2]{BarbotFenley21} for the definition of lozenges
and related terms.
If $\zeta_1, \zeta_2$ are corners orbits of a single
lozenge in this chain, then $\pi(\zeta_i) = \eta_i$
are periodic orbits of $\Phi$ which are freely homotopic
to the inverse of each other as oriented closed curves.
Let $\gamma$ be
a deck transformation representing $\eta_1$ as an oriented
closed curve, and hence $\gamma$ represents $\eta_2^{-1}$. This proves the claimed result. 
\end{proof}

Using this, we will prove a general result about pseudo-Anosov flows regulating to an $\mathbb{R}$-covered foliation.

\begin{proposition}{}\label{conjugate}
Let $\Phi_i$, $i = 1, 2$ be pseudo-Anosov flows 
transverse and regulating to a foliation or
a branched foliation $\cG_i$ in the same hyperbolic
3-manifold $M$. (The foliations $\cG_i$ may be different from
each other.)
Suppose that every periodic orbit of $\Phi_1$ is freely
homotopic to a periodic orbit of $\Phi_2$.
Then $\Phi_1$ is topologically orbit equivalent to $\Phi_2$,
by an equivalence that is homotopic to the identity.
\end{proposition}

\begin{proof}{}
Since $M$ is hyperbolic, then $\Phi_i$ is transitive
\cite{Mosh},
and the union of periodic orbits is dense.
Since the flows are regulating for $\cG_i$, it follows
that the leaf space of $\widetilde \cG_i$ is homeomorphic
to $\R$ for $i = 1,2$.

Suppose that there are distinct periodic orbits $\alpha_1, \alpha_2$
of say $\Phi_1$ which are freely homotopic. By the previous lemma, we can get a deck transformation $\gamma \in \pi_1(M)$ and closed orbits $\eta_1, \eta_2$ of $\Phi_1$ such that they have lifts $\zeta_1, \zeta_2$ to $\mt$ such that $\gamma$ fixes $\zeta_1$ and $\zeta_2$ is a translation along $\zeta_1$ in the flow direction while it acts as a translation on $\zeta_2$ in the direction oposite to the flow. The first implies that $\gamma$ acts increasing on the leaf
space of $\widetilde \cG_1$ and the second implies that $\gamma$
acts decreasingly there, a contradiction. Therefore there is at
most one periodic orbit of $\Phi_1$ in each conjugacy class in $\pi_1(M)$.
Similarly for $\Phi_2$.

Since $M$ is hyperbolic, and $\Phi_1$ is regulating for
a foliation it follows that $\Phi_1$ is {\em quasigeodesic}
\cite{Th2,Fen2012}. This means that orbits are rectifiable
and when lifted to 
the universal cover length along any given orbit of $\widetilde \Phi_1$ 
is uniformly efficient up to a bounded multiplicative
distortion in measuring distance in $\mt$.
In other words if $\delta$ is such an orbit of $\widetilde \Phi_1$,
there is $K > 0$ so that for any $x, y$ in $\delta$ then

$$l_{\delta}(x,y) \ \ \ \leq \ \ \ K d(x,y) + K$$

\noindent
where $l_{\delta}(x,y)$ is the length along $\delta$ from
$x$ to $y$.
In addition there is a uniform bound see 
\cite[Lemma 3.10]{CalegariQG}. 
In other words there is a single $K$ so that the inequality
above works for all orbits.
Using properties of quasigeodesics in hyperbolic manifolds,
this implies that any orbit of $\widetilde \Phi_1$ is a
uniformly bounded distance from the corresponding geodesic
of $\HH^3$ with the same ideal points as this orbit (see e.g. \cite{Th1}).
The same holds for $\Phi_2$.

Let $\oo_i$ be the orbit space of the flow $\widetilde \Phi_i$.
In \cite{Fe-Mo} it is proved that $\oo_i$ is homeomorphic
to the plane $\R^2$. We will produce a homeomorphism
$\tau: \oo_1 \rightarrow \oo_2$ that is $\pi_1(M)$ equivariant.
This will produce the topological equivalence.

For any point $p$ in $\oo_1$ which is a lift of a periodic
orbit $\gamma_1$ of $\Phi_1$, then $\gamma_1$ is freely
homotopic to a periodic orbit $\gamma_2$ of $\Phi_2$. 
We remark that since $M$ is hyperbolic, there is essentially
one free homotopy between the corresponding powers of
$\gamma_1$ and $\gamma_2$ $-$ any two free homotopies
are homotopic.
In addition we proved in the beginning of the
proof that $\gamma_2$ is uniquely determined by $\gamma_1$.
Lift the homotopy from $\gamma_1$ to $\gamma_2$
so that the first orbit
lifts to the orbit in $\oo_1$ corresponding to $p$.
The second orbit lifts to an orbit of $\widetilde \Phi_2$ and
it corresponds to a point $q$ in $\oo_2$.
The second lift is uniquely determined by the uniqueness
of free homotopies above.
In addition since this is a lift of a free homotopy between
closed orbits, then $\gamma_1, \gamma_2$ are a finite
Hausdorff distance from each other. In particular $\gamma_1, \gamma_2$
have the same ideal points (in the sphere at infinity)
in the forward flow direction
and the same ideal point in the backwards flow direction.
This deals with orbits of $\widetilde \Phi_1$ which
are lifts of periodic orbits.

\vskip .1in
We want to extend this to all orbits of $\widetilde \Phi_1$.
Let then $\delta$ be any orbit of $\widetilde \Phi_1$. 
Since the periodic orbits of $\Phi_1$ are dense, there
are orbits $\delta_n$ of $\widetilde \Phi_1$, which are
lifts of periodic orbits and which converge to $\delta$.
Let $\beta_n$ be the corresponding lifts of periodic orbits of
$\Phi_2$. Notice  that any orbit of $\widetilde \Phi_i$ is a
uniformly 
bounded distance from the corresponding geodesic with the same
endpoints at the sphere at infinity.
The orbits $\delta_n$ intersect a fixed compact set
in $\mt$. 
Hence the same is true for the corresponding geodesics by
the above property. Since the orbits of $\Phi_2$ also
satisfy this property, it follows that the orbits
$\beta_n$ intersect a fixed compact set in $\mt$.
So there is a subsequence
$\beta_{n_j}$ which converges to an orbit $\beta$ of
$\widetilde \Phi_2$.
Since all orbits of $\widetilde \Phi_1$ (or $\widetilde \Phi_2$)
are quasigeodesics with uniform constants, then the ideal
points of the quasigeodesics 
$\beta_{n_j}$ also converge to the ideal points
of $\beta$ (c.f. \cite{Th1}).

There is an additional property that we will use:
since $\Phi_i$ is transverse and regulating
for a foliation $\cG_i$, it follows that no two orbits
of $\widetilde \Phi_i$  can
have the same forward and backwards ideal points (notice that for this we need the flow to be pseudo-Anosov and not just quasigeodesic).
This follows from results in
\cite{Fen2012}, we provide some explanations:
Consider $\Phi_1$ and $\cG_1$.
The foliation $\cG_1$ is $\R$-covered. Using \cite[Theorem G]{Fen2012}
we know that $\Phi_1$ does not have perfect fits\footnote{We do not need to
know here what is a perfect fit for an pseudo-Anosov flow.
We only need the statement.}.
With this property, 
\cite[Theorem D]{Fen2012} says that the flow ideal boundary
$\flow$ of $\Phi_1$ is equivariantly homeomorphic to the
ideal sphere of $\mt$.
Now, \cite[Theorem B]{Fen2012} states that $\flow$ is a quotient
of the ideal boundary of the orbit space $\oo_1$ of 
the flow $\widetilde \Phi_1$.
The ideal boundary of $\oo_1$ is a circle $\partial \oo_1$.
By \cite[Theorem B]{Fen2012} two points in $\partial \oo_1$  are sent to the
same point in the flow ideal boundary $\flow$ if and only if
they are either ideal points of the same stable
leaf or the same unstable leaf in $\oo_1$.
Let $\alpha$ be a flow line of $\Phi_1$. Its forward ideal
point in the flow ideal boundary $\flow$ is the ideal
point of the stable leaf in $\oo_1$ containing $\alpha$.
So if two flow lines have the same forward ideal point,
then they are in the same stable leaf.
Similarly if two flow lines have the 
same backwards ideal points then they are in the same unstable
leaf. So if two flow lines share both ideal points, they
are in the same stable and same unstable leaf, which
means that they are the same flow line.
The same holds for $\Phi_2$.

In particular the orbit $\beta$ of $\widetilde \Phi_2$
obtained above
is uniquely
determined by its ideal points.

\vskip .1in
Now we go back to the sequence of orbits $\beta_n$.
Let $p_n, q_n$ be the ideal points of $\beta_n$,
and let $a_n, b_n$ the ideal points of $\delta_n$.
Since $\beta_n$ corresponds to $\delta_n$ they have the
same pair of ideal points.
Up to switching $p_n, q_n$,
then $p_n = a_n, \ q_n = b_n$.
Therefore the ideal points of $\beta$ are the same
as the ideal points of $\delta$.
The analysis and properties of the last two paragraphs implies that
for any subsequence $\beta_{n_k}$ of $\beta_n$
which converges, it has to converge to $\beta$, 
because the ideal points $p_{n_k}, q_{n_k}$
of $\beta_{n_k}$ converge
to the ideal points of $\beta$. In other words,
since any such sequence has a convergent subsequence, then
the whole sequence $\beta_n$ converges to $\beta$.

In addition, more is true:
the same arguments show that for any sequence $\alpha_n$ of lifts of 
periodic orbits of $\Phi_1$ converging to $\delta$, the
corresponding orbits $\epsilon_n$ of $\widetilde \Phi_2$ 
converge to $\beta$ $-$ again this is because of the continuity
property of the ideal points. It follows that $\beta$ is uniquely
defined by the orbit $\delta$. 

This defines a map from $\oo_1$ to $\oo_2$. Again because of
the uniformity of the constants of quasigeodesic behavior, this
map is continuous. In addition, by the property that ideal
points determine the orbit, this map is a bijection onto
its image, which is homeomorphic to $\R^2$. The image
is also $\pi_1(M)$ invariant. Since $\Phi_2$ is transitive
\cite{Mosh},
there is a dense orbit, so this now implies that
the image is $\oo_2$.

\vskip .1in
Using the theory of classifying spaces, Haefliger \cite{Haefliger} 
proved that there is a homotopy equivalence of $M$ sending
orbits of $\Phi_1$ to orbits of $\Phi_2$. Subsequently
Ghys \cite{Ghys} and later
Barbot \cite{Barbot1996} upgraded
this to a homeomorphism that sends orbits to orbits\footnote{This argument seems to have appeared several times in several places. See \cite[Section 8]{HP-survey} for further references.}.
Up to reversing the flow direction, this
homeomorphism also preserves orientation along orbits. 


Since this homeomorphism sends periodic orbits to closed curves
freely homotopic to themselves, it follows that this 
homeomorphism acts trivially on $\pi_1(M)$. Since
$M$ is aspherical, this implies that the homeomorphism
is homotopic to the identity \cite{He}. 
\end{proof}

\subsection{Lefschetz fixed points}\label{ss.Index}  
Recall the Leftschetz formula 
for fixed points.
We refer to the monograph by Franks \cite[Chapter 5]{Franks} or \cite[Appendix C]{BFFP2}. 
for details and more generality. 
Here we only work in the context we shall use.

Let $P$ be a topological plane, $g: P \to P$ be a orientation preserving continuous map and $D \subset P$ a compact disk such that $\mathrm{Fix}(g) \subset \mathrm{int}(D)$. 

The Lefschetz index $\mathrm{Ind}(g)$ of $g$ in $P$ is defined to be the intersection number of the graph of $g$ with the diagonal in $P \times P$. 

We will use the following facts about the Lefschetz index (see \cite[Section 8.6]{KH} for more details): 

\begin{itemize}
\item If all fixed points of $g$ are isolated then $\mathrm{Ind}(g)$ is the sum of all the indices of each fixed point. Moreover, a (topologically) hyperbolic fixed saddle point is $-1$ and the index of a $p$-prong fixed point ($p \geq 3$) is $1-p$. 

\item If $g,h: P \to P$ are orientation preserving maps for which there is $R>0$ so that $d(g(x),h(x))< R$ and there is a disk $D \subset P$ so that for every $x \notin D$ one has that $d(g(x),x) > 2R$ then $\mathrm{Ind}(g) = \mathrm{Ind}(h)$. 

\end{itemize}

\subsection{Triple translation}\label{ss.Triple}

This section will use the setting of Proposition \ref{p.transl2} which is what is left to prove to complete the proof of Theorem B. Recall that we can work up to finite covers and iterates.

\vskip .1in
\noindent
{\bf {Assumption in subsection \ref{ss.Triple}}} The diffeomorphism $f$ is a non-accessible partially hyperbolic diffeomorphism of a hyperbolic 3-manifold $M$ and $f$ is a double translation (c.f. Theorem \ref{teo.clasif}). Moreover, a good lift $\ft$ of $f$ acts as a translation in the lamination $\widetilde \Lambda^{su}$ (which therefore completes to an $\RR$-covered foliation, see Corollary \ref{lami}).

\medskip

The assumption above implies that we can assume that
the diffeomorphism 
$f$ preserves branching foliations
$\fbs, \fbu$ and $\ft$ acts as a translation on both
leaf spaces of $\wfbs, \wfbu$.
Here $\fbs, \fbu$ are branching foliations, but are approximated
arbitrarily closely by actual foliations (as in Theorem \ref{teo-BI}).

Recall the foliation $\fol$ which is an extension of
the lamination $\Lambda^{su}$.
In addition in this case there is a homeomorphism $h$ homotopic
to $f$ so that $h$ preserves $\fol$.
There is a good lift $\widetilde h$ of $h$ to $\mt$,
so that $\widetilde h$ also acts as a translation on
the leaf space of $\widetilde \cF$, which is also $\R$.
Recall that $h = f$ when restricted to $\Lambda^{su}$,
and likewise for corresponding lifts to $\mt$.
Recall that $f$ does not a priori preserve $\fol$ in
the complement of $\Lambda^{su}$ (Corollary \ref{lami}).
This case is then called {\em triple translation}.
Up to a finite cover and iterate if necessary,
we may assume that $\fol$ is transversely orientable.
By Theorem \ref{teo.regul} there is a pseudo-Anosov flow
$\Phi_{su}$ which is transverse to $\fol$ and which is
regulating for $\fol$. This means that every orbit
of $\widetilde \Phi_{su}$ intersects every leaf of
$\widetilde \fol$.

We know that $f$ is a double translation meaning that $\widetilde f$ also translates $\widetilde \cW^{cs}_{bran}$ and $\widetilde \cW^{cu}_{bran}$ (which are therefore also $\RR$-covered and uniform).

Denote by $\Phi_{cs}$ and $\Phi_{cu}$ corresponding regulating pseudo-Anosov flows transverse to $\fbs, \fbu$ respectively. 
In particular, there are 3 pseudo-Anosov flows:
$\Phi_{cs}, \Phi_{cu}$ and $\Phi_{su}$.

We now state a consequence of  \cite[Proposition 10.2]{BFFP2} (see \cite[Proposition 11.1]{BFFP2} for a stronger statement):

\begin{proposition}\label{fixed}
For each $\gamma \in \pi_1(M)$ associated to a periodic orbit of $\Phi_{cs}$ (resp. $\Phi_{cu}$) there exists $z\in \mt$ which verifies that $\gamma^i \circ \widetilde f^k (z) = z$ for some $k \in \mathbb{Z}\setminus \{0\}$. 
\end{proposition}

In particular since $\ft$ acts freely on the leaf space
of $\wfbs$, then both $i, k$ in Proposition \ref{fixed} are 
non zero. To prove this result, in \cite{BFFP2} we use the transverse pseudo-Anosov flow to make an index argument, but we also use crucially the fact that we have information of the transverse dynamics to (say) $\fbs$: 
transversely we have the unstable foliation $\cW^u$ and
$f$ expands unstable length.
This is the reason why we cannot apply this result
directly to $\Lambda^{su}$ (even if $\Lambda^{su}$ were a true foliation):
transverse to $\Lambda^{su}$ we have the center bundle $E^c$.
But $f$ can contract, expand or leave invariant $E^c$ lengths
depending on the particular point.

We will need some properties relating the lifted flow
$\widetilde \Phi_{su}$ with the map $\widetilde h$. Recall
that $\widetilde h$ is an extension of $\ft|_{\widetilde \Lambda^{su}}$
to $\mt$, and $\widetilde h$  preserves $\widetilde \fol$. 
Let $L$ be a leaf of $\widetilde \fol$, which could be a leaf
of $\widetilde \Lambda^{su}$ or not. Then $\widetilde h$ induces
a map from $L$ to $\widetilde h(L)$. This maps any point $x$ to a 
point a bounded distance from $x$ it in $\mt$, because $\widetilde
h$ is a good lift of $h$. 

We define another map $\mu: L \to \widetilde h(L)$: for each
$x$ in $L$, consider the $\widetilde \Phi_{su}$ flow line
and define $\mu(x)$ to the intersection of this flow line
with $\widetilde h(L)$, that is: 

$$  x \in L \ \mapsto \ \mu(x) \ := \ \widetilde \Phi_{su}(x) \cap \widetilde h(L). $$

Since $\Phi_{su}$  is regulating for
$\fol$, there is always such an intersection and this is unique so this map is well defined.
This map $\mu$ depends on $L$, but for notational simplicity
we will omit this dependence. 
In addition since $\fol$ is uniform, then the distance in
$\mt$ from $x$ to $\mu(x)$ is uniformly bounded in $\mt$,
see  \cite{Fen2002}.  
Because $\fol$ is $\R$-covered this
implies that the distance in $\widetilde h(L)$
from $\widetilde h(x)$ to 
$\mu(x)$ is uniformly bounded above \cite{Fen2002}.

Another property we need is the following (see \cite[Lemma 8.5]{BFFP}): 

\begin{lemma}\label{l.unbounded}
For every $L \in \widetilde \fol$, $R$ a positive number,
and $\eta \in \pi_1(M)\setminus \{\mathrm{id}\}$ such that $\eta \circ \mu (L) = L$,
  there exists a compact set $K \subset L$ such that if $y \notin K$ then $d_{L}(\eta \circ \mu (y), y) > R. $
\end{lemma}

This is just a property about regulating pseudo-Anosov flows. In a nutshell, the exponential expansion and contraction of the foliations of the pseudo-Anosov flows and the fact that one has some control on the geometry of the intersection of the stable and unstable laminations of the pseudo-Anosov flow with the transverse foliation allow one to show that outside a compact set, points must move arbitrarily a lot. It is illustrating to check this for different lifts of a pseudo-Anosov homeomorphism of a surface (which would be the case of the suspension flow). 

We can now prove one main property we need.

\begin{proposition}{(equivalent flows)} \label{conjugate2}
The flows $\Phi_{cs}$
and $\Phi_{su}$ are topologically equivalent,
by an equivalence homotopic to the identity.
The same holds for $\Phi_{cu}$ and $\Phi_{su}$.
\end{proposition}

\begin{proof}{}
We show the result for $\Phi_{cs}$ and $\Phi_{su}$.
Let $\alpha$ be a periodic orbit of $\Phi_{cs}$
represented by a deck transformation $\gamma$.
By Proposition \ref{fixed} there is $x \in \mt$ with
$\gamma^i \circ \ft^k (x) = x$, and $k > 0$. We want to show that $\gamma^i$ is associated to a periodic orbit of $\Phi_{su}$ so that we can apply Proposition \ref{conjugate}. 

First we want to show that $x$ can be chosen to
be in $\widetilde \Lambda^{su}$. Recall from Theorem \ref{t.HHU} 
that the completion of any complementary region of $\Lambda^{su}$ is an $I$-bundle, the center bundle is uniquely integrable
in this completion,  and any center segment in this completion
must connect the two boundary components in $\widetilde \Lambda^{su}$. 
Suppose then
that $x$ is not in $\widetilde \Lambda^{su}$. Then
$x$ is in a center segment which connects
two boundary leaves of $\widetilde \Lambda^{su}$ and 
this center segment is fixed by $\gamma^i \circ \ft^k$. In particular
if $y$ is an endpoint of the center segment through
$x$ connecting two boundary leaves of $\widetilde \Lambda^{su}$ then $\gamma^i \circ \ft^k(y) = y$.
So we can assume that $x$ is in $\widetilde \Lambda^{su}$.

Let $L$ be the leaf of $\widetilde \Lambda^{su}$ containing
$x$. 
As above consider the maps $\mu, \ft^k$ from $L$ to
$\ft^k(L) = \widetilde h^k(L)$. Then we have maps $\gamma^i \circ \mu$,
$\gamma^i \circ  \ft^k$ from $L$ to itself.
They are a bounded distance $R>0$ from each other.

Moreover, we know that $\gamma^i \circ \ft^k$ has at least one fixed point (since $\gamma^i \circ \ft^k(x) = x$) and it has Lefschetz index equal to $-1$ since it is a hyperbolic fixed point (because the action of $\gamma^i \circ \ft^k$ in $L$ is the composition of a isometry with a diffeomorphism with two hyperbolic bundles, so $x$ is a hyperbolic saddle, c.f. subsection \ref{ss.Index}). 

By Lemma \ref{l.unbounded} we know that the maps $\gamma^i \circ \ft^k$ and $\gamma^i \circ \mu$ move points more than $2R$ outside a compact set, so we are in the setting of subsection \ref{ss.Index}. In particular
$\gamma^i \circ \ft^k$ and $\gamma^i \circ \mu$ have the same
Lefschetz index. Then one can deduce that $\gamma^i \circ \mu$ has a fixed point of negative index. 

But the map $\mu$ is flow along $\widetilde \Phi_{su}$ from 
$L$ to $\ft^k(L)$. Therefore a fixed point of 
$\gamma^i \circ \mu$ produces
a periodic orbit of $\Phi_{su}$ which is represented
by a power of $\gamma$.

Therefore any periodic orbit of $\Phi_{cs}$ is freely
homotopic to a periodic orbit of $\Phi_{su}$. By Proposition \ref{conjugate} this implies that
$\Phi_{cs}$ is topologically equivalent to $\Phi_{su}$ 
by a topological equivalence homotopic to the identity.

This finishes the proof of the proposition.
\end{proof}

Now we show the following:

\begin{proposition}\label{nonsing}
Let $\gamma \in \pi_1(M)$ associated to a periodic orbit of $\Phi_{su}$ 
so that there exists $x \in \mt$ which belongs to a 
leaf $L \in \widetilde \Lambda^{su}$ so that $\gamma^i \circ \widetilde f^k (x) = x$. Then it follows that $x$ is the unique fixed point of $\gamma^i \circ \widetilde f^k$ in $L$ and that the orbit $\gamma$ is associated to a regular periodic orbit of $\Phi_{su}$ (which has index $-1$). 
\end{proposition}

\begin{proof}
The proof is very similar to the proof that one cannot have
a ``mixed" behavior for a hyperbolic partially hyperbolic diffeomorphism
in a hyperbolic $3$-manifold $M$. This is done in  \cite[Section 12]{BFFP2}. 

In that case one had a center foliation and
a stable foliation in such a leaf $L$ and the center foliation may be branched. The analysis here is
much simpler because we have in $L$ a stable and an unstable 
foliation. In particular positive powers of $f$ are known 
to expand length along unstable leaves, as opposed to the
unknown actual behavior of $f$ along center leaves.
We explain the main steps.

As done in the previous propositions let $\mu$ be the
$\widetilde \Phi_{su}$ flow along from $L$ to $\ft^k(L)$.
Then $\gamma^i \circ \ft^k$ has non zero Lefschetz index in $L$, and
so does $\gamma^i \circ \mu$. Therefore $\gamma^i \circ \mu$ has a unique
fixed point $z$ in $L$. Up to taking powers, suppose that
$\gamma$ leaves invariant all prongs of the orbit through
$z$. If $z$ is a $p$-prong orbit, then the Lefschetz index
$\gamma^i \circ \mu$ is equal to $1 - p$, and hence so is the Lefschetz index
of $\gamma^i  \circ \ft^k$ (c.f. subsection \ref{ss.Index}).

Our goal is to prove that $p = 2$.
The leaf $L$ is quasi-isometric to the hyperbolic plane \cite{Candel},
and hence can be canonically compactified with an ideal circle
$\partial L$.
In particular $\gamma^i \circ \mu$ has exactly $2p$ fixed
points in $\partial L$. 
It follows that
the action of $\gamma^i \circ \ft^k(p)$ has exactly $2p$ fixed points
in $\partial L$, which are alternatively attracting (the set $P= \{P_1, \ldots, P_p\}$)
and repelling (the set $N=\{N_1, \ldots, N_p\}$).  This is because $\gamma^i \circ \mu$
and $\gamma^i  \circ \ft^k$ act in exactly the same way on $\partial L$,
as they are maps a bounded distance from each other in $L$ and 
then use Lemma \ref{l.unbounded} (see figure \ref{f.prong}).

\begin{figure}[ht]
\begin{center}
\includegraphics[scale=0.60]{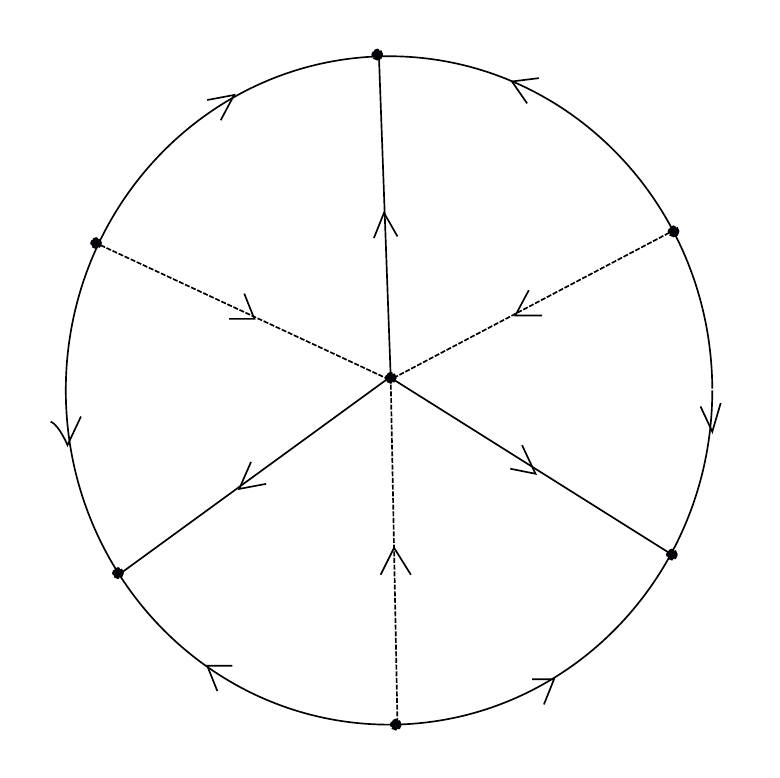}
\begin{picture}(0,0)
\put(-219,154){$N_1$}
\put(-34,166){$N_2$}
\put(-120,4){$N_3$}
\put(-219,52){$P_1$}
\put(-118,216){$P_2$}
\end{picture}
\end{center}
\vspace{-0.5cm}
\caption{{\small The dynamics of $\gamma^i \circ \mu$ in $L$. By Lemma \ref{l.unbounded} the dynamics of $\gamma^i \circ \ft^k$ is very similar near the boundary.}\label{f.prong}}
\end{figure}
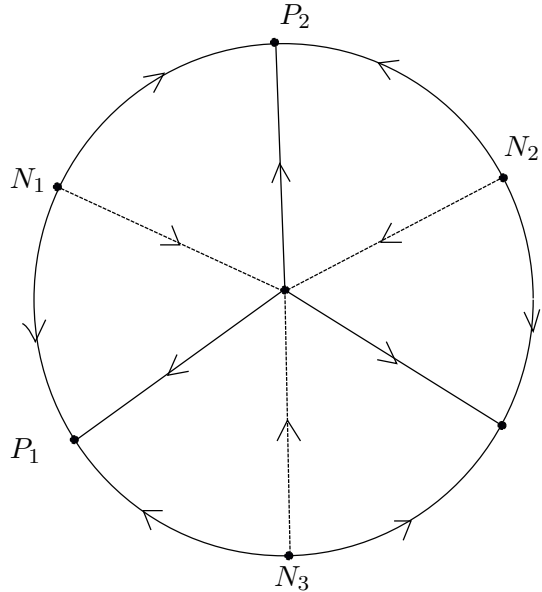

In order to do prove that $p = 2$, we obtain
several facts, whose proof is done in  detail in 
\cite[Theorem 12.1]{BFFP}: 
\begin{itemize}
\item
First, if $x$ is a fixed
point of $\gamma^i \circ \ft^k$ in $L$ and $s(x)$ is a stable ray ending
in $x$, then $s(x)$ can only limit in a single point $w$ in $N$.
\item
An unstable ray through $x$ can only limit in a single point $w$ in $P$.
\item
Since $P, N$ are disjoint the ideal points of periodic stable
rays and periodic unstable rays are disjoint.
\item
Hence the two rays of $s(x)$ limit to distinct points, same for
unstable leaves.
\item
Because the total index is $1-p$, there are exactly $p - 1$
fixed points in $L$.  
\end{itemize}

If $x_1, x_2$ are distinct fixed points of $\gamma^i \circ \ft^k$, then
no ray of $s(x_1)$ can share an ideal point with a ray of $s(x_2)$. This is because if two periodic rays share an ideal point it is possible to either construct a periodic unstable leaf limiting in the repelling point in $N$ or get a configuration as in figure \ref{f.mixed} which gives a contradiction as one is able to find a tangency between stables and unstables. 

\begin{figure}[ht]
\begin{center}
\includegraphics[scale=0.60]{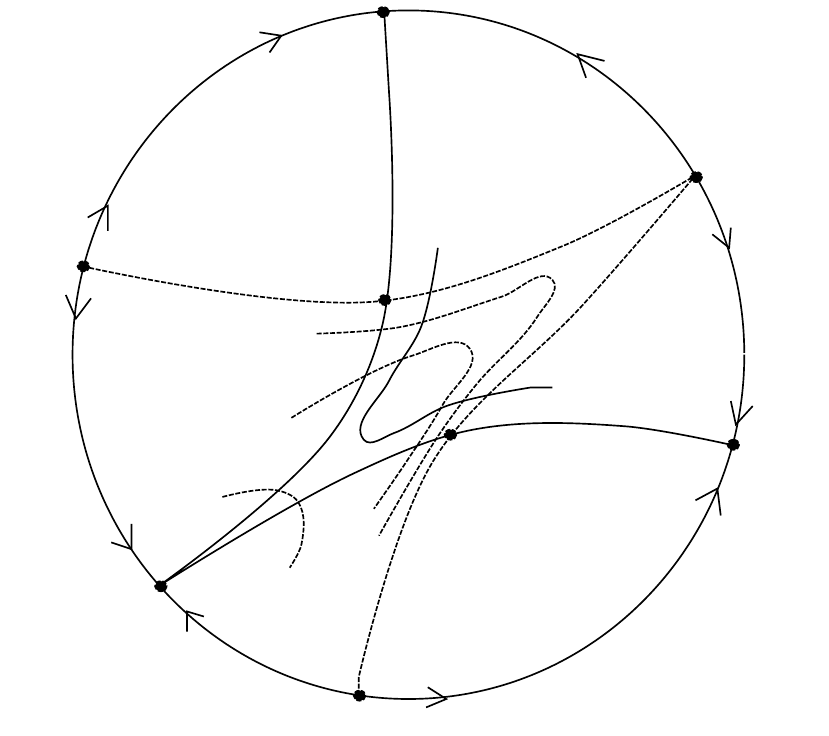}
\begin{picture}(0,0)
\end{picture}
\end{center}
\vspace{-0.5cm}
\caption{{\small If two stable rays limit in the same point either the unstables do not cross from one side to the other and one obtains a periodic unstable limiting in a repelling point or one gets the configuration in the drawing that forces some tangency.}\label{f.mixed}}
\end{figure}

Therefore one gets for each fixed point $x$ in $L$, there are
four fixed points in $\partial L$. 
These collections of $4$ points are pairwise disjoint.
Hence the total is
$4(p-1)$. Since this number is equal to $2p$, it follows that $p = 2(p-1)$
or that $p = 2$.

This finishes the proof of the proposition.
\end{proof}

We mention that there is a simpler proof this result assuming
a result of Mendes \cite{Mendes}.

\subsection{End of the proof of Proposition \ref{p.transl2}}\label{endpf}  
In this section we will assume that $M$ is a hyperbolic 3-manifold.
\vskip .07in

Suppose first that $\Phi_{cs}$ does not have singular orbits.
Then $\Phi_{cs}$ is a topological Anosov flow.
In that case, Theorem \ref{teo.AnosovFlows} implies
that $\pi_1(M)$ is solvable. But this contradicts
that $M$ is hyperbolic.

Suppose now that the 
flow $\Phi_{cs}$ has singular orbits. Let $\alpha$ be a $p$-prong
singular orbit, with $p \geq 3$.
Let $\gamma$ be the deck transformation associated with $\alpha$.
By Proposition \ref{fixed} there is $x$ in $\mt$ with
$\gamma^i \circ \ft^k(x) = x$, and $k > 0$.
In the proof of 
Proposition \ref{conjugate2} we showed that one may assume that
$x$ is in $\widetilde \Lambda^{su}$. 
Let $L$ be the leaf  of $\widetilde \Lambda^{su}$ containing $x$.
By Proposition \ref{nonsing}, there is a periodic orbit
of $\Phi_{su}$ associated with $\gamma$ and it is non singular.
But Proposition \ref{conjugate2} shows that the flows
$\Phi_{cs}$ and $\Phi_{su}$ are topologically equivalent, by
an equivalence homotopic to the identity. This would imply
that the orbit of $\Phi_{su}$ is singular. 
This is a contradiction with Proposition \ref{nonsing} which states
that the orbit of $\Phi_{su}$ in question is a regular 
orbit.

This contradiction shows that the assumption that there is a lamination
$\Lambda^{su}$ is impossible.
It follows that $f$ is accessible. 
This finishes the proof of Proposition \ref{p.transl2} and therefore of Theorem B.

\section{Accessibility without conservative behavior}\label{s.nonconservative}

In this section we obtain a result in the direction of having something like Theorem \ref{t.HHU} without assuming that $f$ is non-wandering. As we explained, the only place in the proofs of accessibility that uses that $f$ is non-wandering is to be able to use the second and third 
conclusions Theorem \ref{t.HHU}. One other context where accesssibility has been established without use of conservativity is \cite[Section 6.3]{HP-Nil} for partially hyperbolic diffeomorphisms in nilmanifolds which are not
the $3$-torus. (In the generic setting, the non-conservative case has also been considered in some references, e.g. \cite{DW,BHHTU}.)

Here we show the following result which directly implies Theorem C. 

\begin{theorem} Let $f: M \to M$ be a partially hyperbolic diffeomorphism of a closed 3-manifold $M$ whose fundamental group is not virtually solvable. Suppose that $f$ is a discretized
Anosov flow. Then $f$ is accessible.
\end{theorem}

We note that in \cite{BPW} stable ergodicity and accessibility of perturbations of certain time one maps of Anosov flows is considered, though in that paper specific properties of the strong foliations of Anosov flows are used crucially. 

\begin{proof} 
We first remark that  \cite[Proposition G.2]{BFFP2}
shows that $f$ is dynamically coherent.

Suppose that $f$ is not accessible.
Then there is a lamination $\Lambda^{su}$ given by the
first conclusion of Theorem \ref{t.HHU}. That only uses that
$f$ is not accessible. Notice that in this case we cannot say anything a priori about the complementary regions of $\Lambda^{su}$
 since $f$ is not necessarily non-wandering. The goal will be to show that 
the completions of these complementary regions are $I$-bundles in order to be able to apply the results of the previous sections. 

Let $\Phi$ be a topological Anosov flow so that $f(x) =
\Phi_{\tau(x)} (x)$ where $\tau(x) > 0$. Let $\ft$ be the lift
preserving flow lines of $\widetilde \Phi$, so $\ft(x) = \widetilde \Phi_{\tau(\pi(x))} (x)$.

Consider a stable leaf $s$ of $f$ lifted to $\mt$. Then all points in $s$
converge together under positive iteration of $\ft$.
It follows that $s$ is contained in a (weak) stable leaf $L$ of $\widetilde \Phi$. 

Any center leaf $c$ in $L$ is an orbit of $\widetilde \Phi$.
We claim that that $s$ intersects $c$. To see this, take a point $x\in s$ and call its center leaf $c_x$. It follows that there exists $n>0$ so that $\ft^n(x)$ is at distance less than $\eps$ from $c$, because they
are in the same weak stable leaf of $\widetilde \Phi$.
By the local product structure, it follows that the stable
leaf $s(\ft^n(x))$ intersects $c$.
Since $\ft^{-n}(s(\ft^n(x)))=s$ and $c$ is $\ft$-invariant it follows that $s$ intersects $c$. Therefore, stable leaves are global sections of the center leaves inside each center stable leaf. These facts are already implicit in \cite{BW} (see in particular \cite[Lemma 3.16]{BW}). Likewise for the center unstable leaves. 

We now consider the orbit space $\mathcal O$ of the lifted
flow $\widetilde \Phi$. This orbit space is homeomorphic to $\RR^2$
(see \cite{FenleyAnosov,BarbotFeuill}).
Take any leaf $E$ of $\widetilde \Lambda^{su}$. Then $E$ is transverse
to $\widetilde \Phi$, as it is transverse to the center bundle.
Hence the projection of $E$ to $\mathcal O$ is locally
injective.
The previous paragraph implies that $E$ projects to $\mathcal O$ as
a stable and unstable saturated set. This implies that
$E$ is a global section of the projection $\mt \rightarrow 
\mathcal O$, in other words a global section for the flow 
$\widetilde \Phi$.

It now follows that the completion of 
any complementary region of $\widetilde \Lambda^{su}$
is an $I$-bundle. In addition the center/flow foliation is a product in
the completion of this complementary region. Therefore the
same is true when projecting to $M$.

This recovers the second conclusion
of Theorem \ref{t.HHU}. 
We do not quite get the third conclusion of Theorem \ref{t.HHU},
that is, the unique integrability of the center bundle
in completions of complementary regions of $\Lambda^{su}$. However,
since $f$ is dynamically coherent, there is a center foliation
and the center leaves form an $I$-bundle structure
in any such completion, as proved above. 
One can now proceed as in sections \ref{s.FL} and \ref{ss.DAF} to get a contradiction and show that $f$ is accessible.
\end{proof} 

Notice that this result is strictly larger than the previous as it has been shown in \cite{BG} that there are examples of partially hyperbolic diffeomorphisms which are discretized Anosov flows and which have proper attractors
and therefore cannot be non-wandering. 

\section{Other isotopy classes on Seifert manifolds}\label{s.pASeif}

In this section we let $M$ be a Seifert 3-manifold with hyperbolic base. Since we are interested in accessibility and the hypothesis is that $f$ is non-wandering, we have already explained that there is no loss of generality on working in a finite cover and taking a finite iterate (cf. Lemma \ref{fmin2}). 

Since every Seifert manifold with hyperbolic base has a finite cover that is an orientable circle bundle over an orientable surface of genus $\geq 2$, we will assume all along that $M$ is a circle bundle over such a surface. We denote by $p: M \to \Sigma$ the circle bundle projection. Every diffeomorphism $f: M \to M$ induces via $p$ an isotopy class of diffeomorphisms of $\Sigma$ because $f$ has to preserve the Seifert fibration up to isotopy.
We say that $f$ induces an element $\rho(f)$ on the mapping class group of $\Sigma$. 

The following result implies Theorem D and is what we will prove in this section:

\begin{theorem}\label{t.SeifPA}
Let $M$ be a circle bundle over a surface and $f: M\to M$ be a partially hyperbolic diffeomorphism whose non-wandering set is $M$ and such that $\rho(f)$ has a pseudo-Anosov component with a $p$-prong with $p\geq 3$. Then, $f$ is accessible. 
\end{theorem}

Given a mapping class $\rho$ of $\Sigma$, there is a canonical decomposition of $\Sigma$ into subsurfaces with boundary on which, up to an iterate, $\rho$ acts either as the identity or leaves no simple closed curve not homotopic to the boundary invariant. The latter pieces are called \emph{pseudo-Anosov}.  We assume the reader has some familiarity with the classification of surface homeomorphisms (see e.g. \cite{HandelThurston}). 

Notice that if $\rho(f)$ is pseudo-Anosov, then it must have at least one $p$-prong with $p\geq 3$ because $\Sigma$ has genus $\geq 2$.

The proof of Theorem \ref{t.SeifPA} 
is very similar to the triple translation case in hyperbolic 3-manifolds.

\begin{proof}
Assume by contradiction that $f$ is not accessible. Then, by Theorem \ref{t.HHU} we get a $f$-invariant non-empty lamination $\Lambda^{su}$ tangent to $E^s \oplus E^u$ which can be extended to a foliation $\cF$. 

We claim that this foliation is horizontal in the sense of \cite{Brit}. Indeed, if it were not horizontal, since the completions
of the complementary regions of $\Lambda^{su}$ are $I$-bundles, then $\Lambda^{su}$ would contain a vertical sublamination. 
Since the base surface is hyperbolic, it follows that
the sublamination is not all of $M$. Hence the boundary leaves for such laminations would be periodic by $f$. Notice that such a vertical leaf is not Gromov hyperbolic, so the proof of Proposition \ref{p.fixedleaf} cannot be applied directly.  Nevertheless, the argument can be carried without problem in this setting too since the leaves are cylinders with
bounded geometry (bounded injectivity radius). Notice that a strong stable cannot intersect the same strong unstable twice due to the graph transform argument alluded to in the proof of Proposition \ref{p.fixedleaf}. This implies that the cylinder has both strong stable and strong unstable leaves going from end to end of the cylinder and that allows to 
apply the argument of \cite[Section 5.4]{HaPS} to get a contradiction. This shows that $\cF$ must be horizontal.  

Since $f$ is non-wandering, using \cite[Proposition 5.1]{HaPS} we know that the branching foliations $\cW^{cs}_{bran}$ and $\cW^{cu}_{bran}$ are horizontal too. 

Consider $\hat M$ to be the cover of $M$ obtained as $\widetilde M/_{<c>}$ where $\widetilde{M}$ is the universal cover of $M$ and $c \in \pi_1(M)$ is associated to a circle fiber. It follows that $\hat M$ is homeomorphic to $\widetilde \Sigma \times S^1$. 

Pick a lift $\hat f$ of $f$ to $\hat M$ associated to the $p$-prong of the induced action $\rho(f)$ in the universal cover $\widetilde \Sigma =\mathbb{H}^2$ of $\Sigma$. This lift has an even number of fixed points at infinity (in $\partial \widetilde \Sigma \sim S^1$) which are alternatingly attracting and repelling in the boundary. The number of fixed points is $\geq 6$ since the $p$-prong is with $p \geq 3$. 

We first need to show that $\hat f$ fixes some leaf of $\cW^{cs}_{bran}$ up to taking an iterate. This can be proven following the same scheme as for \cite[Proposition 8.1]{BFFP}  and it is done in \cite{BFFP3}. 

Let $L = \hat f^k (L)$ with $L \in \cW^{cs}_{bran}$. By an index argument, one can see that $\hat f$ has a fixed point in $L$. As in the proof of Proposition \ref{conjugate2} we get a leaf $F$ of $\Lambda^{su}$ which is fixed. Now, exactly the same argument of Proposition \ref{nonsing} works to give a contradiction (notice that in this case the proof is much simpler since we do not need to compare the large scale action of $\hat f$ on leaves aof $\cW^{cs}_{bran}$ and $\Lambda^{cs}$ because being both horizontal, they mimic the action of the chosen lift in the base). 

\end{proof}

\medskip
\medskip

{\em Acknowledgements:} We are grateful to Jana Rodriguez Hertz and Raul Ures who explained to us the ergodicity
problem. We thank Andy Hammerlindl for several comments (including the proof of Lemma \ref{double}) and explaining \cite{HRHU}. 
Several comments by the referee were important to improve the presentation of the paper.  
SF was partially supported by 
grant award number $280429$ of the Simons foundation.
RP was partially supported by CSIC no.618, FVF-2017-111 and
FCE-1-2017-1-135352 and the work was completed when RP was serving as a Von Neumann Fellow at the Institute for Advanced
Study (IAS), he is grateful to the Minerva Research Foundation and NSF (DMS-1638352) for the support and to IAS for the amazing work environment.

\end{document}